\documentclass{amsart}
\usepackage{amsmath}
\usepackage{amssymb}
\usepackage[all]{xy}
\usepackage{comment}

\theoremstyle{plain}
\newtheorem{thm}{Theorem}[section]
\newtheorem{thm*}{Theorem}[section]
\newtheorem{cor}[thm]{Corollary}
\newtheorem{prop}[thm]{Proposition}
\newtheorem{lemma}[thm]{Lemma}

\newtheorem{lemma*}{Lemma}

\theoremstyle{definition}
\newtheorem{defn}[thm]{Definition}
\newtheorem{remark}[thm]{Remark}

\newtheorem*{remark*}{Remark}
\newtheorem{ex}[thm]{Example}
\newtheorem{notation}[thm]{Notation}

\newtheorem{question*}{Question}
\numberwithin{equation}{thm}

\SelectTips{eu}{}

\newcommand{\bM}{\mathbb M}

\newcommand{\bT}{\mathbb T}

\newcommand{\cN}{\mathcal N}
\newcommand{\cU}{\mathcal U}

\def\Rad{\operatorname{Rad}\nolimits}
\def\Spec{\operatorname{Spec}\nolimits}
\def\Ker{\operatorname{Ker}\nolimits}

\def\Lie{\operatorname{Lie}\nolimits}
\def\Soc{\operatorname{Soc}\nolimits}
\def\Coker{\operatorname{Coker}\nolimits}

\def\Im{\operatorname{Im}\nolimits}

\def\dim{\operatorname{dim}\nolimits}

\def\Grass{\operatorname{Grass}\nolimits}
\def\LG{\operatorname{LG}\nolimits}
\def\Coh{\operatorname{Coh}\nolimits}

\def\Ad{\operatorname{Ad}\nolimits}

\def\Max{\operatorname{Max}\nolimits}
\def\Min{\operatorname{Min}\nolimits}
\def\Id{\operatorname{Id}\nolimits}

\newcommand{\cKer}{\mathcal K\text{\it er}}
\newcommand{\cIm}{\mathcal I\text{\it m}}
\newcommand{\cCoker}{\mathcal C\text{\it oker}}
\newcommand{\bSoc}{\mathbb S\rm oc}
\newcommand{\bRad}{\mathbb R\rm ad}

\newcommand{\bG}{\mathbb G}

\newcommand{\cO}{\mathcal O}
\newcommand{\bA}{\mathbb A}

\newcommand{\cF}{\mathcal F}
\newcommand{\bP}{\mathbb P}
\newcommand{\bZ}{\mathbb Z}

\newcommand{\cC}{\mathcal C}

\newcommand{\cL}{\mathcal L}

\newcommand{\cE}{\mathcal E}

\newcommand{\bW}{\mathbb W}

\newcommand{\fp}{\mathfrak p}

\newcommand{\fg}{\mathfrak g}
\newcommand{\fh}{\mathfrak h}

\newcommand{\fu}{\mathfrak u}

\newcommand{\gl} {\mathfrak {gl}}

\newcommand{\fsl} {\mathfrak {sl}}

\newcommand{\fso} {\mathfrak {so}}
\newcommand{\fsp} {\mathfrak {sp}}

\newcommand{\fl}{\mathfrak l}

\newcommand{\ol}{\overline}

\def\pr{\operatorname{pr}\nolimits}

\def\Spec{\operatorname{Spec}\nolimits}
\def\sl2{\operatorname{SL_{2(2)}}\nolimits}
\def\Ga2{\operatorname{\mathbb G_{\rm a(2)}}\nolimits}
\def\SL{\operatorname{SL}\nolimits}
\def\GL{\operatorname{GL}\nolimits}
\def\PGL{\operatorname{PGL}\nolimits}
\def\Sp{\operatorname{Sp}\nolimits}

\def\SO{\operatorname{SO}\nolimits}

\def\End{\operatorname{End}\nolimits}
\def\Hom{\operatorname{Hom}\nolimits}

\newcommand{\wt}{\widetilde}

\newcommand{\Z}{\mathbb Z}

\newcommand{\bE}{\mathbb E}
\newcommand{\bV}{\mathbb V}

\date\today

\begin{document}

 \title[Vector bundles ]{Vector bundles associated to Lie algebras}
 
 \author[Jon F. Carlson, Eric M. Friedlander, and Julia Pevtsova]
{Jon F. Carlson$^*$, Eric M. Friedlander$^{**}$, and 
Julia Pevtsova$^{***}$}

\address {Department of Mathematics, University of Georgia, 
Athens, GA}
\email{jfc@math.uga.edu}

\address {Department of Mathematics, University of Southern California,
Los Angeles, CA}
\email{ericmf@usc.edu, eric@math.northwestern.edu}

\address {Department of Mathematics, University of Washington, 
Seattle, WA}
\email{julia@math.washington.edu}

\thanks{$^*$ partially supported by the NSF grant DMS-1001102}
\thanks{$^{**}$ partially supported by the NSF grant DMS-0909314 and DMS-0966589}
\thanks{$^{***}$ partially supported by the NSF grant DMS-0800930 and DMS-0953011}



\begin{abstract}  
We introduce and investigate a functorial construction which associates
coherent sheaves  to  finite dimensional (restricted) representations of 
a restricted Lie algebra $\fg$.  These are sheaves 
on locally closed subvarieties of the projective variety $\bE(r,\fg)$ of 
elementary subalgebras of $\fg$ of dimension $r$.  We show that 
representations of constant radical or socle rank studied in \cite{CFP3} which generalize modules of constant Jordan type 
lead to algebraic vector bundles on $\bE(r,\fg)$. For $\fg = \Lie(G)$, the Lie algebra
of an algebraic group $G$,  rational representations of $G$ enable us to realize 
familiar algebraic vector bundles on $G$-orbits of $\bE(r, \fg)$. 
\end{abstract}

\maketitle

\section{Introduction}

In \cite{CFP3}, the authors introduced the projective 
algebraic variety $\bE(r,\fg)$ of {\it elementary
subalgebras} $\epsilon$ of $\fg$ of dimension $r$ 
of a given finite dimensional restricted Lie algebra $\fg$ 
over an algebraically closed field $k$ of characteristic 
$p > 0$.    An elementary subalgebra $\epsilon$ of $\fg$
is  a commutative Lie subalgebra  restricted to which the 
$p$-th power operator $(-)^{[p]}: \fg \to \fg$ is trivial.  In this paper,
we explore the connections between geometric structures 
on these varieties $\bE(r,\fg)$ and 
restricted representations of $\fg$.

We recall that the category of restricted representations 
of $\fg$ is equivalent to the category of 
modules for the finite dimensional associative $k$-algebra 
$\fu(\fg)$, the restricted enveloping algebra
of $\fg$.   We construct coherent sheaves and algebraic 
vector bundles on $\bE(r,\fg)$ associated 
to finite dimensional $\fu(\fg)$-modules,  extending 
considerations in \cite{FP3}
(the case $r=1$) and \cite{CFP2} (the special case 
in which $\fg$ is itself an elementary Lie algebra).  

For a given finite dimensional $\fu(\fg)$-module 
$M$ and a given $r > 0$, we consider radicals and 
socles of $M$ restricted to an elementary subalgebra $\epsilon$ of $\fg$  
as $\epsilon \in \bE(r,\fg)$ varies.   
In Section ~\ref{sec:cohsheaves} we show that for any locally closed subvariety 
$X \subset \bE(r, \fg)$, any $j>0$ and any $\fu(\fg)$-module $M$, there are sheaves 
$\cKer^{j,X}(M)$, $\cIm^{j, X}(M)$ on $X$ which depend functorially on $M$ and whose 
generic fiber is identified naturally with the $j$-th socle or the $j$-th radical of 
$M$. Consequently, these ``image" and ``kernel" sheaves
 encode  considerable information about the action of 
$\fg$ on $M$, with local input the action of  elementary
subalgebras $\epsilon$ on $M$.   Much is known 
about the modules for an elementary Lie algebra (the category of 
which is equivalent to the more familiar category 
of $k(\bZ/p\bZ^{\oplus r})$-modules), even though this
category is wild if $r >1$, unless $r=2, p =2$.

We present two different but equivalent constructions of the sheaves $\cKer^{j,X}(M), \cIm^{j, X}(M)$. 
Our first construction 
uses equivariant descent and natural operators on coherent 
sheaves on the Stiefel variety which is a $\GL$-torsor over a Grassmannian. 
The second construction  involves a straight-forward patching technique, making use of the 
standard affine charts on the Grassmannian and allowing for an easy identification of the generic fiber. 
It is the patching construction that allows us to show in Section~\ref{sec:geominv} 
how modules of constant Jordan type
(and, more generally, modules of constant 
$r$-radical rank and constant  $r$-socle rank) lead to vector
bundles on $\bE(r,\fg)$.
\sloppy{

}

As in \cite{FP3} and \cite{CFP3}, our constructions yield detailed  
 ``new" invariants for $\fu(\fg)$-modules.   
For example, if $\fg = \Lie(G)$ for an affine 
algebraic group $G$ and if $M$ a rational $G$-module, 
$\cIm^{j,X}(M), \ \cKer^{j,X}(M)$ are $G$-equivariant 
algebraic vector bundles provided that 
$X \subset \bE(r,\fg)$ is a $G$-orbit.   
One can ask which coherent sheaves and 
vector bundles can be realized as image and kernel sheaves
of $\fu(\fg)$-modules.  Computations are difficult,
which is to be expected granted the subtleties which
already arise in the case in which $\fg$ is itself 
an elementary Lie algebra as seen in \cite{CFP2}.

A second type of application should arise from  
the explicit nature of our construction of coherent sheaves
from the data of a finitely generated $\fu(\fg)$-module.  
For various types of $\fu(\fg)$-modules $M$ and
for certain subvarieties $X \subset \bE(r,\fg)$, we 
obtain vector bundles; since $X$ is typically singular,
such explicit constructions should provide insight 
into the difficult challenge of understanding 
algebraic vector bundles on singular varieties.   

Yet another application is the explicit construction 
of  familiar vector bundles such as tangent and
cotangent bundles on certain projective varieties 
in terms of  $ \cKer^{j,X}(M)$, $\cIm^{j,X}(M)$, and other similarly constructed sheaves.
In Section \ref{sec:examples}, we investigate 
examples arising from rational modules for an affine algebraic
group, whereas in Section \ref{sec:Gsemi} we 
provide further examples which do not arise from actions
of an algebraic group.

\vspace{0.1in}
The paper is organized as follows. In Section \ref{sec:cohsheaves} 
we present our two constructions of the image and kernel sheaves and show that they are equivalent.
Following the construction, we suggest in Section \ref{sec:geominv}  methods 
to extract geometric invariants for a $\fu(\fg)$-module $M$ arising from our image and 
kernel sheaves on $\bE(r,\fg)$ for various $r$.  The challenge,
which appears to lend itself to only incremental 
progress, is to search among all the geometric data
one obtains for computable invariants which distinguish 
many classes of modules and suggests families
of modules worthy of further study.

If $M$  is the restriction of a rational $G$-module, 
the action of $G$ on $\bE(r,\fg)$ equips
the coherent sheaves  $\cIm^{j,X}(M)$ and 
$\cKer^{j,X}(M)$ on a $G$-stable subvariety $X \subset \bE(r,\fg)$
with the structure of  $G$-equivariant coherent 
sheaves on $X$.  
In  Section \ref{sec:Gorbit}, we focus on the 
context in which $X = G\cdot \epsilon \subset \bE(r,\fg)$ 
is a $G$-orbit and $M$ a rational $G$-module; 
in this case, the image and kernel sheaves  $\cIm^{j,X}(M)$ 
and $\cKer^{j,X}(M)$ are $G$-equivariant algebraic 
vector bundles on  $X$.   If the orbit map 
$\phi_\epsilon: G \to \bE(r,\fg)$ is separable so that
$G\cdot \epsilon$ is isomorphic to $H = G/G_\epsilon$, 
 then we identify in Theorem \ref{thm:functor}  
the vector bundles $\cIm^{j,X}(M), \ \cKer^{j,X}(M)$ on $G \cdot \epsilon$ 
as the $H$-equivariant vector bundles obtained by
induction starting with the representations of $H$
on $\Rad^j(\epsilon^*M)$, $\Soc^j(\epsilon^*M)$, 
where $\epsilon^*M$ denotes the restriction of $M$ to  $\epsilon$. 
Using this identification, we realize many familiar 
vector bundles as image and kernel bundles 
associated to rational $G$-modules.  Hence, for $\fg = \Lie(G)$
and $M$ a rational $G$-module, we get explicitly determined
algebraic vector  bundles associated to $M$ which we can
view as invariants of $M$ for each choice 
of $r > 0, \ X \subset \bE(r,\fg)$ and $j$, such that $1 \leq j \leq (p-1)r$.

The final section of this paper is devoted to vector bundles
which arise from the semi-direct product of an 
algebraic group $H$ with a vector group 
associated to a rational $H$-module $W$.    
We consider image and kernel
bundles for (non-rational) representations 
of $\fg_{W,H} = \Lie(W \rtimes H)$.
Many of the examples of our recent paper 
\cite{CFP2} are reinterpreted and extended 
using this construction.  As we show in Theorem 
\ref{thm:H} and its corollary,
most homogeneous bundles on $H$-orbits inside 
$\Grass(r,W) \subset \bE(r,\fg_{W,H})$
are realized as image bundles in this manner.

Throughout, $k$ is an algebraically closed field 
of characteristic $p>0$.
All Lie algebras $\fg$ considered in this paper 
are assumed to be finite dimensional
over $k$ and $p$-restricted; all Lie subalgebras 
$\fh \subset \fg$ will be
assumed to be closed under $p$-restriction.  
We use the terminology ``affine algebraic group"
to refer to a reduced group scheme represented 
by a finitely generated, integral $k$-algebra
$k[G]$.  We use the terminology ``rational 
representation" of an affine group scheme $G$ to 
mean a comodule for the coalgebra $k[G]$; we 
shall sometimes refer to such rational representations
informally as $G$-modules.   Without explicit 
mention to the contrary, all $G$-modules are finite dimensional.

We thank Burt Totaro for providing a reference necessary 
for simplifying our geometric assumptions 
in Section \ref{sec:Gorbit} and thank George McNinch for helpful
discussions about separability of orbit maps.  
We are especially grateful to the referee for a careful reading of our paper.

%%%%%%%%%%%%%%%%%%%%%%%%%%%%%%%
%%%%%%%%%%section 1%%%%%%%%%%%%

\section{The coherent sheaves 
$\cIm^{j,X}(M), \ \cKer^{j,X}(M)$ on $\bE(r,\fg)$}
\label{sec:cohsheaves}

Let $\fg$ be a 
restricted Lie algebra of dimension $n$. We recall
from \cite[1.3]{CFP3} that an {\it elementary} subalgebra of $\fg$ 
is an abelian restricted Lie subalgebra with trivial $p$-restriction.  
For some $r, \ 0 < r < n$, we consider the 
Grassmann variety $\Grass(r,\fg)$ of 
$r$-planes in $\fg$ which is viewed as an $n$-dimensional 
vector space over $k$.   
The subset of $\Grass(r,\fg)$ consisting of 
those $r$-planes $\epsilon \subset \fg$ which 
are elementary subalgebras  
constitute a closed subvariety: \ 
$\bE(r,\fg) \ \subset \ \Grass(r,\fg).$ \
Let $\bM_{n,r}$ denote the affine space of 
$n \times r$ matrices over $k$ and let $\bM_{n,r}^\circ$ be the 
open subset consisting of those matrices
that have maximal rank. 

For each finite dimensional $\fu(\fg)$-module $M$
and each $j, \ 1 \leq j \leq (p-1)r$, we construct the coherent sheaves 
$\cIm^{j,X}(M), \ \cKer^{j,X}(M)$ on $\bE(r,\fg)$.
The reader should keep in mind that 
these image and kernel sheaves are not images 
and kernels of the action of $\fu(\fg)$,
but rather globalizations of images and kernels 
of  local (with respect to the Zariski 
topology on $\bE(r,\fg)$) actions on $M$.

Indeed, we provide two independent constructions.  
The first is by equivariant descent
for the $\GL_r$ torsor $\xymatrix@-0.6pt{\bM_{n,r}^\circ \ar[r]& \Grass(r,\fg)}$ 
(the Stiefel variety over $\Grass(r,\fg)$).
In the special case $r=1$ (so that $\bG_m$ replaces 
$\GL_r)$, this is implicit in the 
original construction of vector bundles for infinitesimal 
group schemes given in \cite{FP3}.
Each construction has
its advantages: that of equivariant descent is 
quickly seen to be well defined
independent of choices, that of patching leads 
to an identification of fibers.

We employ the natural action of $\GL_r$ on $\fg^{\oplus r}$ given by
$(a_{i,j}) \in \GL_r(k)$ acting as $(a_{i,j})\otimes \Id$ on $k^{\oplus r} \otimes \fg$.  This action induces 
an action of $\GL_r$ on $\fg^{\times r}$,
the affine variety associated to $\fg^{\oplus r}$ 
(isomorphic to the affine space $\bA^{nr}$).
We set $( \fg^{\times r})^o \subset \fg^{\times r}$ 
to be the open subvariety of those $r$-tuples of 
elements of $\fg$ that are linearly independent.  
We denote by \ $\cN_p(\fg) \subset \fg$ \
the closed subvariety of $\fg$ (viewed as affine 
$n$-space) consisting of those $X\in \fg$ with
$X^{[p]} = 0$.  We further denote by \ 
$\cC_r(\cN_p(\fg)) \ \subset \ (\cN_p(\fg))^{\times r}$ the
closed subvariety of $r$-tuples $(X_1,\ldots,X_r)$ 
which are pairwise commuting (as well as $p$-nilpotent),
and by $\cC_r(\cN_p(\fg))^o$ the subset of those 
$r$-tuples that are also linearly independent. 

We consider the following diagram of quasi-projective 
varieties over $k$ with  Cartesian (i.e., pull-back) squares.
We have 
 \begin{equation}
 \label{diag}
\xymatrix{ 
\cC_r(\cN_p(\fg)) \ar[r] & \fg^{\times r} \\
\cC_r(\cN_p(\fg))^o \ar[d] \ar[r] \ar[u] & (\fg^{\times r})^o \ar[d] \ar[u]\\
\bE(r,\fg) \ar[r] & \Grass(r,\fg)
 }\end{equation}
where upper vertical maps are open immersions,  
lower vertical maps are quotient maps 
by the $\GL_r$ actions, and  horizontal maps are closed immersions.

We choose a basis $\{ x_1,\ldots,x_n \}$ of $\fg$ which 
determines an identification $k[\mathbb M_{n,r}] \simeq 
k[\fg^{\times r}]$, where $\mathbb M_{n,r}$ is the vector 
group (isomorphic to $\bA^{nr}$) of $n\times r$ matrices.  
Under this identification, the matrix function $T_{i,s} 
\in k[\mathbb \bM_{n,r}]$ is sent to $x_i^{\#} \circ \pr_s:
\fg^{\times r} \to k$ defined as first projecting to the 
$s^{\rm th}$ factor and then applying the linear dual of $x_i$.
We set $Y_{i,s}\in k[\cC_r(\cN_p(\fg))]$  to 
be the image of the matrix function 
$T_{i,s}$ under the surjective map 
$$\xymatrix@=16pt{k[\bM_{n,r}] \ \simeq k[\fg^{\times r}] \ 
\ar@{->>}[r]& k[\cC_r(\cN_p(\fg))]}, \quad T_{i,s} \mapsto Y_{i,s}.
$$
For  any $s, \ 1 \leq s \leq r$, we define 
\begin{equation}
\label{thetass} 
\Theta_s \ \equiv \ \sum_{i=1}^n x_i \otimes Y_{i,s} \, 
\in \, \fg \otimes  k[\cC_r(\cN_p(\fg))]
\end{equation}
and use the same notation to denote the operator 
$$
\Theta_s: M \otimes k[\cC_r(\cN_p(\fg))] \ \to 
\ M \otimes k[\cC_r(\cN_p(\fg))]  ,
\quad 
\Theta_s(m \otimes f) = \sum\limits_{i=1}^n  x_im \otimes  Y_{i,s} f 
$$ 
for any finite dimensional $\fu(\fg)$-module $M$.
 
\begin{prop}
\label{prop:indept}
The operator $\Theta_s$ of (\ref{thetass}) does 
not depend upon the choice of basis of $\fg$.
\end{prop}

\begin{proof}
Let $\{ y_1,\ldots,y_n \}$ be another choice of basis of $\fg$, 
and set $Z_{i,s}$ equal to the image of 
$T_{i,s}$ under the surjective map 
$k[\bM_{n,r}] \to k[\cC_r(\cN_p(\fg))]$ determined by this choice.  
Let $(a_{i,j}) \in \GL_n(k)$ be the change of basis matrix, 
so that $y_j = \sum_i a_{i,j} x_i$.   
Since ${Y_{i,s}}$'s are the images of the linear duals 
to ${x_i}$'s under the projection $k[\bM_{n,r}] \to k[\cC_r(\cN_p(\fg))]$ 
(and similarly for $Z_{i,s}$),  we conclude that
$Z_{j,s} = \sum_i b_{j,i}Y_{i,s}$ where $(b_{i,j}) = (a_{i,j})^{-1}$.  
To prove the proposition, it suffices to 
observe that 
$$
\sum_j y_j \otimes Z_{j,s} \ \equiv \ \sum_j(\sum_i a_{i,j}x_i) 
\otimes (\sum_i b_{j,i} Y_{i,s}) \ 
= \ \sum_i x_i\otimes Y_{i,s}.
$$
This follows directly from the fact that $(a_{i,j}) \cdot (b_{i,j})$ 
is equal to the identity matrix.
\end{proof}

Let $j: X \subset \bE(r,\fg)$ be a locally closed embedding, and denote by 
$\wt X \to X$ the restriction of the $\GL_r$-torsor 
$\cC_r(\cN_p(\fg))^o \to \bE(r,\fg)$ to $X$ 
so that there is a Cartesian square 
\begin{equation}
\label{eq:wtX}
\xymatrix{\wt X \ar@{^(->}[r]^-{\wt j} \ar[d] & \cC_r(\cN_p(\fg))^o \ar[d]\\
X \ar@{^(->}[r]^-j & \bE(r,\fg). } 
\end{equation}
We specialize (\ref{thetass}) by defining 
\begin{equation}
\label{thetaX}
\Theta_s^{\wt X}: M \otimes \cO_{\wt X} \ \to \ M \otimes \cO_{\wt X} ,
\quad 
\Theta_s^{\wt X}(m \otimes f) = \sum\limits_{i=1}^n  x_im \otimes  \wt j^*(Y_{i,s}) f.
\end{equation}

\begin{defn}
\label{defn:descent} 
For any finite-dimensional $\fu(\fg)$-module $M$, 
and any $j$, with $1 \leq j \leq (p-1)r$, 
we define the following submodules of 
$M \otimes k[\cC_r(\cN_p(\fg))]$: 
\[\Im\{\Theta^j,M\} = \Im \{\sum\limits_{\sum j_\ell=j}
\Theta_1^{j_1}\cdots \Theta_r^{j_r}: 
(M \otimes k[\cC_r(\cN_p(\fg))])^{\oplus r(j)} 
\to M\otimes k[\cC_r(\cN_p(\fg))] \}, 
\] 
\[
\Ker\{\Theta^j,M\} = \Ker \{[\Theta_1^{j_1}\cdots \Theta_r^{j_r}]_{\sum j_\ell=j}: 
M \otimes k[\cC_r(\cN_p(\fg))] \to 
(M \otimes k[\cC_r(\cN_p(\fg))])^{\oplus r(j)} \},
\] where $r(j)$ is the number of ways 
to write $j$ as a sum of $r$ nonnegative integers. 

For any locally closed subset $X \subset \bE(r, \fg)$, 
we define  the following coherent sheaves on $\wt X$:
\[
\Im\{\Theta^{j,\wt X},M\} = \Im \{\sum\limits_{\Sigma j_\ell=j}
(\Theta_1^{\wt X})^{j_1}\cdots (\Theta_r^{\wt X})^{j_r}: 
(M \otimes  \cO_{\wt X})^{\oplus r(j)} \to M \otimes  \cO_{\wt X}  \}, 
\]
\[
\Ker\{\Theta^{j,\wt X},M\} = \Ker 
\{[(\Theta_1^{\wt X})^{j_1}\cdots (\Theta_r^{\wt X})^{j_r}]_{\Sigma j_\ell=j}: 
M \otimes \cO_{\wt X} \to (M \otimes \cO_{\wt X}^{\oplus r(j)} \}. 
\]
\end{defn}

\begin{remark}
\label{rem:indept}
By Proposition \ref{prop:indept}, $\Im\{\Theta^{j,\wt X},M\}, \ 
\Ker\{\Theta^{j,\wt X},M\} $ do not depend upon 
our choice of basis for $\fg$.
\end{remark}

Let $G$ be an affine algebraic group (such as $\GL_r$)
 and $X$ an algebraic variety on which $G$ acts.  A quasi-coherent sheaf $\cF$ 
 of $\cO_X$-modules is said to be $G$-equivariant 
if there is an algebraic (i.e., functorial with respect to
base change from $k$ to any finitely generated commutative $k$-algebra $R$) 
action of $G$ on $\cF$ compatible with the action of $G$ on $X$: for all open
subsets $U \subset X$ and every $h, g \in G(R)$, an $\cO_X(U_R)$-isomorphism 
${}^g(-): \cF(U_R) \to \cF(U^{g^{-1}}_R)$ such that 
${}^h(-) \circ {}^g(-) = {}^{gh}(-)$.  
This is equivalent to the following data: an isomorphism $\theta: \mu^*\cF \
\stackrel{\sim}{\to} \ p^*\cF$ 
(where $\mu, \ p: G \times X \to X$ are the action and projection maps)
 together with a cocycle condition on the 
pull-backs of $\theta$ to $G \times G \times X$ 
insuring that  ${}^h(-) \circ {}^g(-) = {}^{gh}(-)$.

The argument of \cite[Lemma 6.7]{CFP2} applies  without change to show the following: 

\begin{lemma}
\label{le:equiv} Let  $M$ be a $\fu(\fg)$-module. 
For any locally closed subset $X \subset \bE(r, \fg)$,  
$\Ker\{\Theta^{j, \wt X},M\}$, $\Im\{\Theta^{j, \wt X},M\}$ 
are $\GL_r$-invariant $\cO_{\wt X}$-submodules of $M \otimes \cO_{\wt X}$.
\end{lemma}

The relevance of Lemma~\ref{le:equiv} to our 
consideration of coherent sheaves on $\bE(r,\fg)$
becomes evident in view of the following categorical equivalence.

\begin{prop}
\label{prop:categ}
There is a natural equivalence of categories 
\begin{equation}
\label{eq:equiv1} 
\xymatrix{\eta: \Coh^{\GL_r}(\cC_r(\cN_p(\fg))^o) 
\ar[r]^-{\sim}& \,\Coh(\bE(r,\fg))
}
\end{equation}
between the category of $\GL_r$-equivariant 
coherent sheaves on $\cC_r(\cN_p(\fg))^o$ and 
the category of coherent sheaves on $\bE(r,\fg)$. 

Moreover, (\ref{eq:equiv1}) restricts to 
an equivalence of categories
\begin{equation}
\label{eq:equiv2} 
\xymatrix{\eta_X: \Coh^{\GL_r}(\wt X) \ar[r]^-{\sim}& \,\Coh(X)}
\end{equation}
for any locally closed subset $X \in \bE(r, \fg)$ 
and $\wt X \to X$ as in \eqref{eq:wtX}. 
\end{prop}

\begin{proof}
This follows from the observation that 
$\cC_r(\cN_p(\fg))^o \to \bE(r,\fg)$ is a $\GL_r$-torsor.  
See, for example, \cite[6.5]{CFP2}.
\end{proof}

Proposition \ref{prop:categ} immediately gives our construction of
image and kernel sheaves.

\begin{thm}
\label{thm:equiv3}
Let $M$ be a finite dimensional $\fu(\fg)$-module, 
let $X \subset \bE(r,\fg)$ be a locally
closed subvariety, and let $j$ be a positive integer with $j \leq (p-1)r$.  Then 
the $\GL_r$-invariant $\cO_{\wt X}$-submodules of $M \otimes \cO_{\wt X}$
of Lemma \ref{le:equiv}, 
$$Im\{\Theta^{j, \wt X},M\},\quad \Ker\{\Theta^{j, \wt X},M\}$$ 
determine coherent subsheaves of $M \otimes \cO_X$  on $X$:
$$\cIm^{j,X}(M),\quad \cKer^{j,X}(M)$$
via the categorical equivalence $\eta_X$ of \eqref{eq:equiv2}.
\end{thm}

Let $G$ be an affine algebraic group over $k$.  
Then the structure of a rational $G$-module on a 
$k$-vector space $M$ is the data of a functorial 
action of $G(R)$ on $M\otimes R$ for all 
finitely generated commutative $k$-algebras $R$.    
This readily implies that if $G$ acts on 
an algebraic variety $X$ and if $M$ is a 
finite dimensional rational $G$-module, then $M \otimes \cO_X$ is a 
$G$-equivariant coherent $O_X$-module with $G$ acting diagonally on the tensor product.

If $\fg = \Lie(G)$ is the Lie algebra of an 
affine algebraic group, then $G$ acts on 
$\fg^{\times r}$ by the diagonal adjoint action 
and this action commutes with that 
of $\GL_r$.   This observation leads to the 
following refinement of Theorem \ref{thm:equiv3}.

\begin{cor}
\label{cor:equiv}
Let $G$ be an affine algebraic group, 
$\fg = \Lie(G)$, $X \subset \bE(r,\fg)$ a $G$-stable locally
closed subvariety, and $M$ a  rational $G$-module.  
Then  $\cIm^{j,X}(M)$ and $ \cKer^{j,X}(M)$
are $G$-equivariant sheaves on $X$ for any $j$ with $1 \leq j \leq (p-1)r$.
\end{cor}

\begin{proof}
The diagonal action of $G$ on $\fg^{\times r}$ 
determines an action of $G$ on $\cC_r(\cN_p(\fg))^o$
and thus on $\wt X \subset \cC_r(\cN_p(\fg))^o$ 
over the $G$-stable subvariety $X \subset \bE(r,\fg)$.  
We thus may consider the category of 
$G$-equivariant coherent sheaves on $\wt X$.
If $M$ is a rational $G$-module, then the 
maps (\ref{thetaX}) are maps of $G$-equivariant
coherent sheaves on $\wt X$; consequently, 
$\Im\{ \Theta^{j,\tilde X},M \}, \ \Ker\{ \Theta^{j,\tilde X},M \}$ of 
Definition \ref{defn:descent} are 
$G$-equivariant coherent sheaves on $\wt X$.
Since the action of $\GL_r$ on  these 
$G$-equivariant coherent sheaves on $\wt X$
commutes with this action of $G$, the 
equivalence $\eta_X$ in (\ref{eq:equiv2}) 
(given by descent) sends these
$G$-equivariant coherent sheaves on $\wt X$ 
to $G$-equivariant coherent sheaves on $X$.
\end{proof}

The special case of Corollary \ref{cor:equiv} 
in which $X$ is a $G$-orbit is of particular interest 
since any $G$-equivariant sheaf on such an $X$ 
is a $G$-equivariant vector bundle. 

\vskip .1in

We now proceed to identify these image and kernel
subsheaves of $\cO_X \otimes M$ when restricted
to open subsets $X \cap \cU_\Sigma$, where 
 $\{ \cU_\Sigma \}$ is a standard affine open 
covering of $\Grass(r,n)$ such that the 
$\GL_r$-torsor $\rho: \bM_{n,r}^\circ \to \Grass(r,n)$ 
splits over each $\cU_\Sigma$.  

Once again, we choose
a basis $x_1,\ldots,x_n$ for $\fg$ as a 
$k$-vector space, thereby identifying $\Grass(r,\fg)$ with
$\Grass(r,n)$.  Let $\Sigma \subset \{1,\ldots,n \}$ 
range over the subsets of cardinality $r$. For a given $\Sigma$, let 
$\rho^{-1}(\cU_\Sigma) \subset \bM_{n,r}^\circ$ be a subset 
of those $n\times r$ matrices whose $r\times r$-minor
with rows indexed by elements 
of $\Sigma$ has non-vanishing determinant.  Thus,
$\cU_\Sigma \subset \Grass(r,n)$ consists 
of those $r$-planes in $n$-space which project 
onto $r$-space via the map which forgets 
the coordinates not indexed by elements of $\Sigma$.
We define the section 
$$
s_\Sigma: \cU_\Sigma \ \to \ \rho^{-1}(\cU_\Sigma)
$$
by sending an $r$-plane $L \in \cU_\Sigma$ to 
the unique $n\times r$-matrix $\wt L$ satisfying 
the conditions that $\rho(\wt L) = L$ and 
that the $r\times r$-minor of $\wt L$ with rows indexed
by elements of $\Sigma$ is the identity 
matrix $I_r$.  In particular, $\rho^{-1}(\cU_\Sigma) \to \cU_\Sigma$
is a trivial $\GL_r$-torsor for any $\Sigma$.

\begin{notation}
\label{note3}
For $\Sigma = \{i_1, \dots, i_r\}$ with $i_1 < \dots < i_r$, $s_\Sigma$
provides an identification of $k[\cU_\Sigma]$ with the quotient
\begin{equation}
\label{eq:YSigma}
k[\bM_{n,r}] = k[T_{i,j}]_{1\leq i \leq n,  1 \leq j \leq r } \  
\longrightarrow \
k[Y_{i,j}^\Sigma]_{i \notin \Sigma, 1\leq j \leq r} = k[\cU_\Sigma]
\end{equation}
sending $T_{i,j}$  to 1, if $i = i_j  \in \Sigma$; 
to 0 if $i=i_{j^\prime} \in \Sigma$ and $j \not= j^\prime$; 
and to $Y_{i,j}^\Sigma$ otherwise.  
For notational convenience,  we set $Y_{i,j}^\Sigma$  equal to 1, if
$i\in \Sigma$ and $i = i_j$, and we set 
$Y_{i,j}^\Sigma = 0$ if $i=i_{j^\prime} \in \Sigma$ 
and $j \not= j^\prime$.

As in (\ref{thetass}), we define
\begin{equation} 
\label{eq:Theta}
\Theta_s^\Sigma \equiv \sum_{i=1}^n x_i  
\otimes Y_{i,s}^\Sigma: M \otimes k[\cU_\Sigma]  
\to M \otimes k[\cU_\Sigma], 
\end{equation} 
by
$$
m \otimes 1 \mapsto  \sum_i x_i(m) \otimes  Y^\Sigma_{i,s}. 
$$

For any closed subset $i: W \subset \Grass(r, \fg) \simeq \Grass(r,n)$,
set $W_\Sigma = W \cap \cU_\Sigma$ and 
$$Y_{i,j}^{W, \Sigma} \ = \ i^*(Y_{i,j}^\Sigma).$$ 
\end{notation}
We define
\begin{equation} 
\label{eq:ThetaW}
\Theta_s^{W,\Sigma} \equiv \sum_{i=1}^n x_i  
\otimes Y_{i,s}^{W,\Sigma}: M \otimes k[W_\Sigma]  
\to M \otimes k[W_\Sigma], 
\end{equation} 
by
$$
m \otimes 1 \mapsto  \sum_i x_i(m) \otimes  Y^{W,\Sigma}_{i,s}. 
$$

\begin{defn}
\label{defn:localj} 	Let $M$ be a finite
dimensional $\fu(\fg)$-module, let 
$W\subset \bE(r,\fg)$ be a closed subset, 
let $\Sigma \subset \{ 1, \ldots, n\}$ be a 
subset of cardinality $r$, and choose $j$ such that $1 < j \leq (p-1)r$.  
We define the following $k[W_\Sigma]$-submodules 
of the free module $M \otimes k[W_\Sigma]$: 
\[
\cKer^{j,W_\Sigma}(M) \  \equiv \ 
\Ker\{ \bigoplus\limits_{j_1+\cdots +j_r = j} 
(\Theta_1^{W,\Sigma})^{j_1}\ldots
(\Theta_r^{W,\Sigma})^{j_r}: M \otimes k[W_\Sigma]  
\to (M \otimes k[W_\Sigma])^{\oplus r(j)} \}
\]
\[
\cIm^{j,W_\Sigma}(M) \  \equiv \ 
\Im\{ \sum\limits_{j_1+\cdots +j_r = j} (\Theta_1^{W,\Sigma})^{j_1}\ldots
(\Theta_r^{W,\Sigma})^{j_r}: (M \otimes k[W_\Sigma])^{\oplus r(j)}   
\to M \otimes k[W_\Sigma]\}
\]
where $r(j)$ is the number of ways $j$ can be written 
as a sum of $r$ nonnegative integers,
$j = j_1 + \cdots + j_r$.

We identify these $k[W_\Sigma]$-submodules of $M\otimes k[W_\Sigma]$ with
coherent subsheaves of the free $\cO_W$-module $M \otimes \cO_W$ restricted
to the affine open subvariety $W_\Sigma \subset W$.
\end{defn} 

\begin{thm}
\label{thm:compare}
Let $M$ be a $\fu(\fg)$-module, $X  \subset 
\bE(r,\fg)$ be a locally closed subset, 
$W = \overline X$ be the closure of $X$, and $r,j$ be positive
integers with $j \leq (p-1)r$.   Then 
$$\cIm^{j,X}(M) \subset M \otimes \cO_X\quad \text{restricted to} \ X \cap \cU_\Sigma \subset X$$
equals 
$$\cIm^{j,W_\Sigma}(M) \subset  M  \otimes \cO_{W_\Sigma}  \quad  \text{restricted to} \
X \cap \cU_\Sigma \subset W_\Sigma.$$

The analogous identification of restrictions of $\cKer^{j,X}(M) \subset M \otimes  \cO_X$ 
are also valid.
\end{thm}

\begin{proof}
It suffices to show that the asserted equalities of 
subsheaves of  $M \otimes \cO_X$ on $X$ are valid 
when restricted to each open chart $\cU_\Sigma \cap X$ 
of $X$ as $\Sigma$ runs through subsets 
of cardinality $r$ in $\{ 1, 2, \ldots ,n \}$.  
Moreover, it suffices to verify the equality of subsheaves of 
$M \otimes \cO_{\wt X}$ on $\wt X$ obtained by pulling back
these restrictions along the map 
$\bM_{n,r}^\circ \to \Grass(r,n)$.  These equalities are verified by comparing
the formulation of $\Theta_s$ in (\ref{thetass}) with that of 
$\Theta_s^\Sigma$ in (\ref{eq:Theta}); namely, $\Theta_s^\Sigma$
is the restriction along the section $s_\Sigma: \cU_\Sigma \to \rho^{-1}(\cU_\Sigma)$
of $\Theta_s$.
\end{proof}

The following  proposition identifies the 
``generic" fibers of the image and kernel sheaves.
  This is particularly 
useful when the locally closed subset  $X \subset \bE(r,\fg)$
is an orbit closure. Here and throughout the paper, we 
denote by $\epsilon^*M$ the restriction of a $\fg$-module $M$ 
to the subalgebra $\epsilon \subset \fg$.

\begin{prop}
\label{prop:fibers}
Let $M$ be a $\fu(\fg)$-module, $X  \subset 
\bE(r,\fg)$ be a locally closed subset, 
$W = \overline X$ be the closure of $X$, and $r,j$ be positive
integers with $j \leq (p-1)r$.  For any $\Sigma \subset \{ 1, \ldots, n\}$ 
of cardinality $r$ there exists an open dense subset 
$U \subset X \cap W_\Sigma$ such that for any point  $\epsilon  \in U$
with residue field $K$
there are natural identifications 
$$
\cIm^{j,X}(M)_\epsilon = \cIm^{j}(M)_{W_\Sigma}\otimes_{k[W_\Sigma]} 
K \ = \ \Rad^j(\epsilon^*(M_K)),
$$
$$
\cKer	^{j,X}(M)_\epsilon = 
\cKer^j(M)_{W_\Sigma}\otimes_{k[W_\Sigma]} K \ 
= \ \Soc^j(\epsilon^*(M_K)).
$$
\end{prop}

\begin{proof} Since $X$ is open dense 
in $W$, we may assume that $W=X$. 
For $\epsilon \in W_\Sigma$ a generic point, the given 
identifications are immediate consequences of
the exactness of localization and 
Definition  \ref{defn:localj}. 
The fact that these identifications 
apply to an open subset
now follows from the generic flatness 
of the $k[W_\Sigma]$-modules 
$\cIm^j(M)_{W_\Sigma}, \  \cKer^j(M)_{W_\Sigma}$.
\end{proof}

 \begin{remark}
For an elementary example of the failure of the isomorphism 
$\cKer^j(M)_\epsilon \ \simeq \ \Soc^j(\epsilon^*M)$
outside of an open subset of $W_{\Sigma}$, we consider
$\fg = \fg_a \oplus \fg_a,$ $r=1$, and 
$j=1$. Let $\{x_1, x_2\}$ be a fixed basis of $\fg$, and
let $M$ be the four dimensional module with basis $\{m_1, \dots, m_4\}$,
such that $x_1m_1 = m_4, \  x_1m_2 = x_1m_3 = x_1m_4 = 0$ and 
$x_2m_1 = m_3, \ x_2m_2 = m_4, \ x_2m_3 = x_2m_4 = 0$.    
We can picture $M$ as follows:
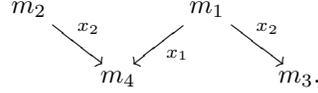
\[
\xymatrix@-.8pc{
m_2\ar[dr]^{x_2}&& m_1 \ar[dr]^{x_2}\ar[dl]^{x_1}&\\
&m_4&&m_3.
}
\]
The kernel of 
$$ x_1 \otimes 1 + x_2 \otimes T_2^{\{1\}}: M \otimes k[T_2^{\{1\}}] 
\to M  \otimes k[T_2^{\{1\}}]$$
(as in Definition \ref{defn:localj} ) with $j=1$
is a  free $k[T_2^{\{1\}}]$-module of rank 2, 
generated by $m_3 \otimes 1$ and $m_4 \otimes 1$.  
The specialization of this module at
the point $\epsilon = kx_1$ (letting
$T_2 \to 0$)  is a vector space of dimension 2. This is a 
proper subspace of $\Soc(\epsilon^*(M))$ which is spanned by
$m_2, m_3, m_4$. 
\end{remark}
 
\begin{remark} In this paper, we concentrate on the variety 
of elementary subalgebras $\bE(r, \fg)$ and  its $G$-orbits. 
Nevertheless, the formalism of the equivariant descent 
construction of the sheaves $\cIm^{X}(M)$, $\cKer^{X}(M)$ 
works equally well for any locally closed subvariety $X \subset \Grass(r, \fg)$, 
  since the commutativity condition 
that defines $\bE(r, \fg)$ does not enter into the construction 
of the image and kernel sheaves for $j=1$. Moreover, one can 
easily check that the identification of the image and kernel 
sheaves on the affine charts as in Theorem~\ref{thm:compare}  
and identification of the generic fibers as in Proposition~\ref{prop:fibers} 
remain valid.  
\end{remark}

%%%%%%%%%%%%%%  Section 2 %%%%%%%%%%%%%%%%%
%%%%%%%%%%%%%%%%%%%%%%%%%%%%

\section{Geometric invariants of $\fu(\fg)$-modules}
\label{sec:geominv}

In this section, we discuss various invariants 
of finite dimensional $\fu(\fg)$-modules $M$ 
which involve consideration of the projective 
varieties $\bE(r,\fg)$ of elementary subalgebras
of $\fg$.  We begin by recalling various closed 
subvarieties of $\bE(r,\fg)$ introduced in
\cite{FP3} associated to
$M$, before considering the image and kernel 
sheaves of Section \ref{sec:cohsheaves}.
As always in this paper, $\fg$ denotes a finite 
dimensional restricted Lie algebra over $k$. For a 
Lie subalgebra $\epsilon \subset \fg$, we denote by 
$\epsilon^*M$ the restriction of a $\fg$-module $M$ to $\epsilon$.  

\begin{defn} (\cite[3.2]{CFP3},\cite[3.15]{CFP3})
\label{defn:subvar}
Let $M$ be a finite dimensional $\fu(\fg)$-module, 
$r$ a positive integer, and $j$ an integer 
satisfying $1 \leq j \leq (p-1)r$.   We define the following closed
subvarieties of $\bE(r,\fg)$ associated to $M$.  Let 
\begin{equation}
\label{eq:supp}
\bE(r,\fg)_M \ \equiv \ \{ \epsilon \in \bE(r,\fg): 
\epsilon^*M \text{\ is not projective} \},
\end{equation}
\begin{equation}
\label{eq:rad}
\bRad^j(r,\fg)_M \ \equiv \ \{ \epsilon \in \bE(r,\fg): 
\dim(Rad^j(\epsilon^*M)) < \Max R_j\},
\end{equation}
and 
\begin{equation}
\label{eq:soc}
\bSoc^j(r,\fg)_M \ \equiv \ \{ \epsilon \in \bE(r,\fg):  
\dim(Soc^j(\epsilon^*M)) > \Min S_j\}, 
\end{equation} 
where $\Max R_j$ is the maximum value of 
$\dim(\Rad^j(\epsilon^{\prime *}M))$ and 
$\Min S_j$ is the minimum value of 
$\dim(\Soc^j(\epsilon^{\prime *}M))$, as $\epsilon^\prime$
ranges over all elements of $\bE(r,\fg)$.
\end{defn}

The closed subvariety $\bE(1,\fg)_M$  equals 
the (projectivized) support
(or, equivalently, rank) variety of $M$ considered 
by various authors (e.g., \cite{FPar1}).  If 
$\bE(1,\fg)_M \neq \bE(1,\fg)$, then 
$$
\bE(1,\fg)_M \ = \ \bRad(1,\fg)_M \ = \ \bSoc(1,\fg)_M.
$$
For $r=1$, these subvarieties were introduced in 
\cite{FP4}; for general $r,j$ they 
were defined in \cite[3.1]{CFP3}.  

The reader is directed to \cite[4.6]{CFP2} for 
an interesting example of a module $M$
for $\fu(\fg)$ with $\fg$ elementary for 
which $\bRad^1(2,\fg)_M \not= \emptyset, \ 
\bSoc^1(2,\fg)_M = \emptyset.$

The following theorem emphasizes the additional 
information given by the image and kernel
sheaves of Section \ref{sec:cohsheaves}.

\begin{thm}
\label{thm:bundle}
Let $M$ be a $\fu(\fg)$-module, $r$ and $j$ be positive 
integers, such that $j \leq (p-1)r$.
Set $Z= \bRad^j(r, \fg)_M$ (resp, $Z= \bSoc^j(r, \fg)_M$) and let
$X = \bE(r,\fg) \ \backslash \ Z$ denote 
the Zariski open subset of $\bE(r,\fg)$ given as 
the complement of $Z$. 

Then $\cIm^{j, X}(M) = \cIm^{j,\bE(r,\fg)}(M)_{|X}$ and 
$\cKer^{j, X}(M) = \cKer^{j,\bE(r,\fg)}(M)_{|X}$  
are algebraic vector bundles on $X$.

Moreover, the fiber of $\cIm^{j,X}(M)$ 
(reps., $\cKer^{j,X}(M)$) at any $\epsilon \in X$ is 
naturally identified with $\Rad^j(\epsilon^*M)$ 
(resp. $\Soc^j(\epsilon^*M)$). 
\end{thm}

\begin{proof} 
It suffices to restrict to an arbitrary 
$\Sigma \subset \{ 1,\ldots, n\}$ of cardinality $r$
and prove that the $\cO_{\cU_\Sigma \cap X}$-modules 
$\cIm^j(M)_{|\cU_\Sigma \cap X}$ 
(resp., $(\cKer^j(M)_{|\cU_\Sigma \cap X}$) are 
locally free.  Here, $\cU_\Sigma \subset
\Grass(r,\fg)$ is as in Notation \ref{note3}.  
We set $W$ equal to the closure of $X$
in $\bE(r,\fg)$.

Set $\Theta_s^\Sigma(\epsilon)$ for  
$\epsilon \in \cU_\Sigma$ equal to  the 
specialization of $\Theta_s^\Sigma$ (as given 
in (\ref{eq:Theta})) at the point  $\epsilon$.
Thus, $\Theta_s^\Sigma(\epsilon)$ is the map   
$\Theta_s^\Sigma \otimes_{k[\cU_\Sigma]} k$
defined by tensoring along evaluation at 
$\epsilon$, $ k[\cU_\Sigma] \to k$.
Since specialization is right exact, 
\[\Coker\{ \sum_{\sum j_i=j}(\Theta_1^\Sigma)^{j_1} \cdots 
(\Theta_r^\Sigma)^{j_r} \} \otimes_{k[\cU_\Sigma]} k
\ = \
\Coker\{ 
\sum_{\sum j_i=j}
\Theta_1^\Sigma(\epsilon)^{j_1}\cdots\Theta_r^\Sigma(\epsilon)^{j_r} \}.
\]
This equals 
$$\Coker\{ \sum_{\sum j_i=j}(\Theta_1^{W,\Sigma})^{j_1} \cdots 
(\Theta_r^{W,\Sigma})^{j_r} \} \otimes_{k[W_\Sigma]} k$$
for $\epsilon \in W_\Sigma = \cU_\Sigma \cap W$.
Exactly as in the proof of \cite[6.2]{CFP2}, 
the hypothesis that $\dim \Rad^j(\epsilon^*M)$ 
is the same for any $\epsilon \in X$  implies that 
$\Coker\{ \sum\limits_{\sum j_i=j}(\Theta_1^\Sigma)^{j_1} 
\cdots (\Theta_r^\Sigma)^{j_r} \}_{|\cU_\Sigma \cap X} $
is a locally free $\cO_{\cU_\Sigma \cap X}$-module.  
The short exact sequence
$$
\xymatrix@-0.8pc{
0 \ar[r]& \cIm^j(M)_{W_\Sigma}\ar[r]& (M\otimes k[W_\Sigma])^{\oplus r(j)}
\ar[r] 
& \Coker\{ \sum\limits_{\sum j_i=j}(\Theta_1^{W,\Sigma})^{j_1} 
\cdots (\Theta_r^{W,\Sigma})^{j_r} \}  \ar[r]& 0}
$$ 
localized at $\cU_\Sigma \cap X$ implies that 
$\cIm^j(M)_{|\cU_\Sigma\cap X}$ is locally free, and also enables 
the identification of the fiber above $\epsilon \in \cU_\Sigma \cap X$. 

The proof for $\cKer^j(M)$ is a minor adaptation of above;
see also the proof of Theorem 6.2 of \cite{CFP2}.
\end{proof}

We recall from \cite[4.1]{CFP3}, that a $\fu(\fg)$-module 
is said to have constant $(r,j)$-radical rank
(respectively, $(r,j)$-socle rank) if the dimension of 
$\Rad^j(\epsilon^*M)$ (respectively, $\Soc^j(\epsilon^*M)$)
is independent of $\epsilon \in \bE(r,\fg)$.
As an immediate corollary of Theorem \ref{thm:bundle}, 
we verify that $\cIm^{j,\bE(r,\fg)}(M)$ 
(respectively, $ \cKer^{j,\bE(r,\fg)}(M)$)
is an algebraic vector bundle on $\bE(r,\fg)$ 
provided that $M$ has constant 
$(r,j)$-radical rank (respectively, constant $(r,j)$-socle rank).

\begin{cor}
\label{cor:bundle}   
Let $M$ be an $\fu(\fg)$-module 
of constant $(r,j)$-radical rank 
(respectively, $(r,j)$-socle rank).  
Then the coherent sheaf $\cIm^{j,\bE(r,\fg)}(M)$ 
(resp., $\cKer^{j,\bE(r,\fg)}(M)$) 
is an algebraic vector bundle on $\bE(r,\fg)$.

Moreover, the fiber of $\cIm^{j,\bE(r,\fg)}(M)$ 
(respectively, $\cKer^{j,\bE(r,\fg)}(M)$) at $\epsilon$ is 
naturally identified with $\Rad^j(\epsilon^*M)$ 
(respectively, $\Soc^j(\epsilon^*M)$).
\end{cor}
\begin{proof}
The condition of constant $(r,j)$-radical rank 
(respectively, $(r,j)$-socle rank) implies that 
$\bRad^j(r, \fg)_M = \emptyset$ 
(respectively, $\bSoc^j(r, \fg)_M = \emptyset$). Hence, 
the corollary is a special case of 
Theorem~\ref{thm:bundle} with $X=\bE(r,\fg)$.
\end{proof}

\begin{ex}
\label{ex:unip}
Let $\fu$ be a nilpotent restricted Lie algebra 
such that $x^{[p]}=0$ for any $x \in \fu$, 
and let $\fu$ also denote the adjoint 
module of $\fu$ on itself. Assume that 
$\fu$ has a maximal elementary subalgebra of 
dimension $r$.  We see that $X  = \bE(r,\fu) \backslash 
\bSoc(r,\fu)_{\fu}$ is  the open subvariety of
$\bE(r,\fu)$ consisting of all maximal elementary 
subalgebras of dimension $r$.   That is, if  
$\epsilon \in X$ is maximal, then
$\Soc(\epsilon^*(\fu)) = \epsilon$, and 
otherwise the dimension of $\Soc(\epsilon^*(\fu))$
is larger than $r$.  Applying Theorem
\ref{thm:bundle} we conclude that    
$\cKer^{1,X}(\fu) \subset \fu\otimes \cO_X$ 
is isomorphic to the restriction along
$X \subset \bE(r,\fu) \subset \Grass(r,\fu)$ of the canonical rank $r$ 
subbundle $\gamma_r \subset \fu \otimes \cO_{\Grass(r,\fu)}$.

If we take $\fu$ to be the Heisenberg algebra $\fu_3$
(the Lie subalgebra of strictly upper triangular matrices in $\gl_3$), then
$\bE(2, \fu_3) \simeq \bP^1$ whenever $p \geq 3$ and every 
$\epsilon \in \bE(2, \fu_3)$ is maximal.  In this case, 
$$
\cKer^{1,\bE(2,\fu_3)}(\fu_3)  \ \simeq \ 
\cO_{\bP^1}(-1) \oplus \cO_{\bP^1} \ 
\subset \ \fu_3 \otimes \cO_{\bP^1}.$$
\end{ex} 

The following proposition refines the analysis given in 
\cite{FP3} of projective modules on $\fsl_2^{\oplus r}$.
We implicitly use the isomorphism $\bE(r,\fsl_2^{\oplus r}) 
\simeq (\bP^1)^{\times r}$ of \cite[1.12]{CFP3}.

\begin{prop}
Let $\fg = \fsl_2^{\oplus r}$ and let $\pi_s: \fg \to \fsl_2$ 
be the $s$-th projection, $1 \leq s \leq r$.  Assume $p \geq 3$.
For each $\lambda$, $0 \leq \lambda \leq p-1$, let 
$P_\lambda$ be the indecomposable projective 
$\fu(\fsl_2)$-module of highest weight $\lambda$.
Then for each  $(\lambda,s) \not= (\lambda^\prime,s^\prime)$,  
there exists some $j$ such that the vector bundle
$\cKer^{j, \bE(r,\fg)}(\pi_s^*(P_\lambda))$ 
on $\bE(r,\fg)$ is not isomorphic to 
$\cKer^{j,\bE(r,\fg)}(\pi_{s^\prime}^*(P_{\lambda^\prime}))$.
\end{prop}

\begin{proof}
Observe that $\Soc^j(\epsilon^*(\pi_s^*M)) = \Soc^j(\epsilon_s^*M)$ 
for any $\fu(\fsl_2)$-module $M$ and
any $j, 1 \leq j \leq (p-1)r$, where 
$\epsilon = (\epsilon_1,\ldots,\epsilon_r) \in \bE(r,\fsl_2^{\oplus r})$; 
in particular,
the action of $\epsilon$ on $\epsilon^*(\pi_s^*M)$ 
factors through $\epsilon_s$.  This implies that
$\cKer^{j,\bE(r,\fg)}(\pi_s^*(P_\lambda)) \ \simeq \pi_s^*(\cKer^{j,\bE(1,\fsl_2)}(P_\lambda))$.  
The proposition now follows from the computation given in \cite[6.3]{FP3}.
\end{proof}

%
%%%%%%%%%%%%%%%%%%%%%%%%%%%%%%%%%%%%%
%%%%%%SECTION ON BUNDLES%%%%%%%%%%%%%%%%%%%%%%%%
%%%%%%%%%%%%%%%%

\section{Vector bundles on $G$-orbits of $\bE(r,\fg)$: constructions}
\label{sec:Gorbit}

Our explicit examples of algebraic vector 
bundles involve considerations of 
image, cokernel, and kernel sheaves associated to a 
rational $G$-module on $G$-orbits of $\bE(r, \fg)$, 
where $G$ is an algebraic group and 
$\fg$ is the Lie algebra of $G$. In this section 
we develop and recall the techniques which allow us to calculate 
some examples in the next section. 
In particular, Theorem \ref{thm:orbit2} verifies that 
the image and kernel sheaves determine
algebraic vector bundles  on $G$-orbits 
inside $\bE(r,\fg)$.  These vector bundles
are interpreted in Theorem \ref{thm:functor} 
in terms of the well-known induction functor from
rational $H$-modules to vector bundles on $G/H$ for an appropriate subgroup $H \subset G$.
This latter theorem is the main computational tool that we apply in Section~\ref{sec:examples}.

Let $\epsilon \in \bE(r,\fg)$ be an elementary 
subalgebra, and let $X = G \cdot \epsilon$ 
be the $G$-orbit of $\epsilon$ in $\bE(r,\fg)$. 
Then $X$  is open in its closure   
and, hence, to any finite-dimensional 
rational $G$-representation $M$ and any $j$, 
$ 1 \leq  j \leq (p-1)r$, we can associate 
coherent sheaves $\cIm^{j,X}(M)$, $\cKer^{j, X}(M)$ 
on $X$ as in Theorem \ref{thm:equiv3}. Let $\cCoker^{j, X}(M)$ 
denote the quotient sheaf $(M \otimes \cO_X)/\cIm^{j, X}(M)$.

\begin{thm}
\label{thm:orbit2} 
Let $G$ be an affine algebraic group, $\fg = \Lie(G)$, 
and $M$ a  rational $G$-module.
Let $\epsilon \in \bE(r,\fg)$ be an elementary 
subalgebra of rank $r$, and let \
$X = G \cdot \epsilon \subset \bE(r,\fg)$   
be the orbit of $\epsilon$ under the 
adjoint action of $G$.  

Then
$$
\cIm^{j,X}(M), \quad  \cKer^{j,X}(M), \quad  \cCoker^{j,X}(M)
$$ 
are algebraic vector bundles on $X$.  

Moreover, we have natural identifications as $H$-modules
$$
\cIm^{j,X}(M)_\epsilon \ \simeq \ \Rad^j(\epsilon^*M), 
\quad 
\cIm^{j,X}(M)_\epsilon  \ \simeq \ \Soc^j(\epsilon^*M),
$$
where $H$ is the (reduced) stabilizer of $\epsilon \in X$.
\end{thm}

\begin{proof}
Since $X$ is a $G$-stable locally closed 
subset of $\bE(r,\fg)$, the coherent sheaves 
$\cIm^{j,X}(M), \ \cCoker^{j,X}(M), \ \cKer^{j,X}(M)$ 
are $G$-equivariant  by  Corollary~\ref{cor:equiv}. 
If $x = g\cdot\epsilon $ for some $g\in G$, then 
the action of $g$ on one of these sheaves sends 
the fiber at $\epsilon$ isomorphically to the 
fiber at $x$.  Since $X$ is  Noetherian, 
we conclude that the sheaves  are locally free 
(see, for example,  \cite[4.11]{FP3} or \cite[5. ex. 5.8]{Har}).

The action of $G$ on $\bE(r,\fg)$ determines for each $g\in G,
\ x \in X$ an isomorphism $g: \cO_{X,x} \to \cO_{X,g^{-1}x}$.
Together with the action of $G$ on $M$, this determines the
(diagonal) action $g: M\otimes \cO_{X,x} \to M\otimes \cO_{X,g^{-1}x}$.
In particular, this determines an action of $H$ on
$M \otimes \cO_{X,\epsilon}$.  Since the action of $G$ is 
$\cO_X$-linear, this determines actions of $H$ on the fibers at
$\epsilon$,
$\cIm^{j,X}(M)_\epsilon, \ \cIm^{j,X}(M)_\epsilon$,
of the coherent sheaves $\cIm^{j,X}(M), \ \cIm^{j,X}(M)$.
As is readily checked using the explicit description of the action
just given, this action on the fibers is that determined by the
action of $H \subset G$ on $M$.  The second assertion now follows
using the isomorphisms of Proposition \ref{prop:fibers}.
\end{proof}

The quotient of an affine algebraic group $G$ by a closed subgroup $H$
is representable by variety $G/H$ (see \cite[I.5.6(8)]{Jan}).  The following
``sheaf-theoretic induction functor" enables a reasonably explicit description
of $G$-equivariant coherent sheaves on $G/H$.

\begin{prop}
\label{prop:cL} 
Let $G$ be an affine algebraic group, 
$H \subset G$ a closed subgroup.   For each
(finite dimensional) rational $H$-module $W$, 
consider the sheaf of $\cO_{G/H}$-modules
$\cL_{G/H}(W)$ which sends an open subset 
$U \subset G/H$ to
\begin{equation}
\label{cL}
\cL_{G/H}(W)(U) \ = \  \{\text {sections of~} G 
\times^H W \to G/H \text{~ above ~} U\}.
\end{equation}
\begin{enumerate}
\item So defined, $W \mapsto \cL_{G/H}(W)$ induces an equivalence
of categories 
\begin{equation} 
\label{l:eq}
\left\{\begin{array}{c}  \text{finite dimensional}\\ \text{rational $H$-modules} 
\end{array}\right\} \quad \sim \quad  \left\{\begin{array}{c}  \text{$G$-equivariant 
algebraic}\\ \text{vector bundles on $G/H$} 
\end{array}\right\}
\end{equation}	
	
\item  If $W$ is the restriction of 
a rational $G$-module, then $\cL_{G/H}(W)$ is isomorphic
 to $W\otimes \cO_{G/H}$, a free coherent sheaf of 
$\cO_{G/H}$-modules.
\item $\cL_{G/H}(-)$ is exact and commutes with 
tensor powers $(-)^{\otimes^i}$, duals $(-)^\#$, 
symmetric powers $S^i(-)$, divided 
powers $\Gamma^i(-)$, exterior powers $\Lambda^i(-)$,
and Frobenius twists $(-)^{(i)}$.
\end{enumerate}
\end{prop}

\begin{proof}
A discussion of the functor $\cL_{G/H}(-)$ of (\ref{cL}) can be found in in \cite[I.5.8, I.5.9]{Jan}
as well as in  \cite[\S 1]{Panin}.

The key property of this functor in our context of $H \subset G$
 is that $\cL_{G/H}(W)$ as in (1) is a locally free sheaf on 
$G/H$.  This follows from the fact that $p: G \to G/H$ is an $H$-torsor 
in the fppf topology (see \cite[III.4.1.8]{DG}) and that
the sheaf of sections of such an 
$H$-torsor is locally trivial in the Zariski topology
by \cite[III.4. 2.4]{DG}.

The fact that $\cL_{G/H}(-)$ is $G$-equivariant and induces the equivalence 
of categories (\ref{l:eq}) is observed in \cite[\S 1]{Panin} where the functors 
inducing this equivalence are made explicit.  Statement (2) easily follows from the equivalence (\ref{l:eq}) 
and is also shown in \cite[5.12.(3)]{Jan} in greater generality than we need here.

To prove (3), we refer the reader to the proof of \cite[II.4.1]{Jan} in the special case 
in which $H \subset G$ is a parabolic subgroup of a reductive algebraic group; 
this restriction on $H\subset G$ is  used only to insure that $\cL_{G/H}(W)$ is 
locally free which is verified above for more general $H \subset G$.
\end{proof}

The following proposition is an essentially immediate corollary of Proposition \ref{prop:cL},
especially (\ref{l:eq}).

\begin{prop}
\label{prop:cL2} Let $G$ be an affine algebraic  
group, and $H \subset G$ be a closed subgroup. 
Let $\cE$ be a  $G$-equivariant vector bundle on $G/H$;
set $W$ equal to  the fiber of $\cE$ 
over the coset $eH \in G/H$ and equip this fiber 
with the $H$-module structure obtained by 
restricting the action of $G$ to $H$ (which stabilizes the fiber over $eH \in G/H$).
Then there is a unique $G$-equivariant isomorphism
$$ \cL_{G/H}(W) \ \quad \simeq \quad \cE$$
which is the identity map on fibers over  the coset $eH \in G/H$.
\end{prop}

We point out the following immediate 
consequence of Proposition \ref{prop:cL2}.

\begin{cor}
\label{cor:cansub}
Let $W$ be a rational representation of $H$ 
of dimension $r$, $V$ a rational representation of $G$,
and $W \subset V$ be a monomorphism of $H$-modules.  Then 
$\cL_{G/H}(W) \ \subset \ \cL_{G/H}(V)=V \otimes\cO_{G/H}$ naturally 
corresponds to a map $f: G/H \to \Grass(r,V)$ sending the orbit $gH$ 
to the subspace $g \cdot W \subset V$.  Under this correspondence, we have
\[\cL_{G/H}(W) = f^*(\gamma_r),\]
where $\gamma_r$ is the canonical rank 
$r$ subbundle on $\Grass(r,V)$. Moreover, the embedding 
$\cL_{G/H}(W) \ \subset \ \cL_{G/H}(V)=V \otimes \cO_{G/H}$ is the pull-back via $f$ of the canonical embedding 
$\gamma_r \ \subset \ V\otimes \cO_{\Grass(r,V)}$.
\end{cor}

\begin{ex}
\label{ex:standard}
We identify some standard bundles using the 
functor $\cL$ in the special case $G = \GL(V) = \GL_n$ and 
$P = P_{r,n-r}$,  a maximal parabolic with 
the Levi factor $L \simeq \GL_r \times\GL_{n-r}$.
Set $X = \Grass(r,V) = G/P$ and denote by 
$W \subset V$ the subspace of dimension $r$ 
stabilized by $P$;
the action of $P$ on $W$ is given by composition 
of the projection $P \to L \to \GL(W).$ 
Then $\cL_{G/P}(W) \subset V\otimes \cO_{G/P}$ corresponds to the
$G$-equivariant isomorphism $f: G/P \stackrel{\sim}{\to} \Grass(r,V)$
sending the identity coset to  $W \subset V$.  Thus, by Corollary 3.4
we have isomorphisms
\begin{equation}
\label{deltagamma}
\gamma_r \ \simeq \ \cL_{G/P}(W) \subset \cL_{G/P}(V) 
\simeq V\otimes \cO_X, \quad \delta_{n-r} \ 
\simeq \ \cL_{G/P}((V/W)^\#),
\end{equation}
where $\gamma_r$ (resp., $\delta_{n-r}$) is the canonical rank $r$ (resp., rank $n-r$) subbundle on $X$.

Observe that we have a short exact sequence of 
algebraic vector bundles on $X$: 
\begin{equation}
\xymatrix{0 \ar[r]& \gamma_{r} \ar[r]& V \otimes  
\cO_X \ar[r] & \delta_{n-r}^\vee \ar[r]& 0 },
\end{equation}
where we denote by $\cE^\vee$ the dual  sheaf to $\cE$. 
If $F(-)$ is one of the functors of Proposition \ref{prop:cL}.(3),
then Proposition \ref{prop:cL}.(3) implies that
$$F(\gamma_r) \ \simeq \ \cL_{G/P}(F(W)).$$
\end{ex}

Combining Theorem \ref{thm:orbit2} and  Proposition 
\ref{prop:cL2}, we conclude 
the following ``identifications" of the vector 
bundles on a $G$-orbit in $\bE(r,\fg)$ associated to a 
rational $G$-module.
The proof follows immediately from these propositions.

\begin{thm}
\label{thm:functor}
Let $G$ be an algebraic group and $M$ be a rational $G$-module.
Set $\fg = \Lie(G)$, and let $r$ be a positive integer.  Let
$X \ \equiv \ G \cdot \epsilon \ \subset \ \bE(r,\fg)$ be a $G$-orbit and
set $H \subset G$ to be the (reduced) stabilizer 
of $\epsilon \in \bE(r,\fg)$.

We assume that $X \simeq G/H$, and consider
$\cL_{G/H}: H{\text-mod} \to 
G/H{\text -bundles}$ as in \eqref{cL}. 
For any $j$, $ 1 \leq j \leq (p-1)r$, we have the 
following isomorphisms of $G$-equivariant vector bundles
 \[\cIm^{j,X}(M) \ \simeq \ \cL_{G/H}(\Rad^j(\epsilon^*M)), 
\quad \cKer^{j,X}(M) \ \simeq \ \cL_{G/H}(\Soc^j(\epsilon^*M))\]
as subbundles of the trivial bundle $\cL_{G/H}(M) = M \otimes \cO_X$, 
where $\Rad^j(\epsilon^*M)$, 
$\Soc^j(\epsilon^*M)$  are endowed with the action 
of $H$ induced by the action of $G$ on $M$. 
\end{thm}

If the orbit map $\phi_\epsilon: G \to G\cdot \epsilon \subset \bE(r,\fg)$ is 
separable, then  $\phi_\epsilon$ induces an isomorphism 
$\ol \phi_\epsilon: G/H \simeq G\cdot \epsilon$ (see, 
for example, \cite{Jan04}). 
To complement Theorem \ref{thm:functor}, we give the following criterion for the
separability of the orbit map.

\begin{thm}
\label{thm:sep}
Let $G$ be a simple algebraic group whose Coxeter number $h$ satisfies the condition $p > 2h-2$.
Then for any  $r$-dimensional subspace  $\epsilon$ of $\fg$ whose (adjoint) orbit 
$G \cdot \epsilon$ is closed in $\Grass(r,\fg)$, the orbit map 
\[
\phi_\epsilon: G \ \to \ G\cdot \epsilon \subset \Grass(r,\fg)
\]
is separable.
\end{thm}

\begin{proof}
Since $p > 2h-2$, the prime $p$ does not divide the order of the finite covering $G \to \Ad(G)$ of $G$ 
over its adjoint form (see \cite[Planche I - X, VI]{Bour}) and thus $p$ does not divide the degree of
the field extension $k(G)/k(\Ad(G))$; consequently, this covering map is separable.
Since $G \cdot \epsilon = \Ad(G)\cdot \epsilon$, we may assume that $G = \Ad(G)$, and, hence,   
that the map $\Ad: G \to \GL(\fg)$ is injective.

The orbit map $\psi_\epsilon: \GL(\fg) \to \Grass(r,\fg)$ is locally trivial (see, for example,
\cite[II.1.10]{Jan}); indeed, this orbit map is a $P$-torsor, where $P \subset \GL(\fg)$
is a standard parabolic subgroup identified here as the stabilizer of $\epsilon$ with respect to
the action of $\GL(\fg)$ on $\Grass(r,\fg)$ given by left multiplication.
We consider the following commutative diagram 
 \begin{equation}
 \label{sq}
\xymatrix{  
P_\epsilon \ar[r]  \ar@{^(->}[d]   & G  \ar[r]^{\phi_\epsilon} \ar@{^(->}[d]^{Ad} & G \cdot \epsilon  \ar@{^(->}[d] \\
P \ar[r] & \GL(\fg) \ar[r]^{\psi_\epsilon} & \Grass(r, \fg)
} \end{equation}
\noindent
where $P_\epsilon \subset G$ is the (reduced) stabilizer of $\epsilon$.

We recall that $G/P_\epsilon$ is quasi-projective since $P_\epsilon$ is a closed 
(reduced) subgroup of $G$ (see \cite[II.6.8]{Bo}).  On the other hand,  $G /P_\epsilon \to G \cdot \epsilon$
is proper (in fact, finite),  and $G\cdot \epsilon \subset \Grass(r,\fg)$ is assumed to be closed and thus
proper over $k$, so that  $G/P_\epsilon$ is proper over $k$ as well as quasi-projective.  Thus,
 $G /P_\epsilon$ is projective which means that   $P_\epsilon \subset G$ is a parabolic subgroup.  
 By construction, $P_\epsilon \ = \ G \cap P$.

To prove that the orbit map $\phi_\epsilon: G \to G\cdot \epsilon$ is separable 
it suffices to show that the tangent map $d\phi_\epsilon$ at the identity 
is surjective (see \cite[4.3.7]{Sp}).  Let $\fp_\epsilon = \Lie P_\epsilon$, and  
let $\fg = \fu^-_\epsilon \oplus \fp_\epsilon$.    We proceed to prove that 
the $k$-linear map of vector spaces
\begin{equation}
\label{eq:comp}
\xymatrix{\fu_\epsilon^- \ar@{^(->}[r] &  \fg \ar[r]^-{(d\phi_\epsilon)_1} & \bT_{\epsilon}(G\cdot \epsilon) } 
\end{equation} 
is injective and thus by dimension reasons an isomorphism.  This will imply the subjectivity
of $d\phi_\epsilon$ at the identity.

Since $\psi_\epsilon$ is separable, $\ker(d\psi_\epsilon)_1 = \Lie(P)$.    Suppose that
$X \in \fu_\epsilon^- \cap \ker(d\phi_\epsilon)_1$.  The commutativity
of \eqref{sq} and the separability of $\psi_\epsilon$ implies that such an $X$ must be in $\Lie(P)$. 
We recall that a subgroup $H \subset \GL_N$  of exponential type in the sense of \cite{SFB1}
has the property that for any $p$-nilpotent $Y \in \Lie(H)$ the 1-parameter subgroup
$\exp_Y: \bG_a \to \GL_N$ factors through $H$.  As essentially observed in \cite{SFB1}, 
 $P \hookrightarrow \GL(\fg)$ is an embedding of exponential type.  Moreover,   
as verified in \cite[7.4.1]{McN}  for $G$ simple of adjoint type and $p > 2h-2$, 
the embedding  $\Ad: G \to \GL(\fg)$ is also of exponential type.  Thus, 
$\exp_X: \bG_a \to \GL(\fg)$ factors (uniquely) via some 1-parameter subgroup
$e_X: \bG_a \ \to \ G \cap P = P_\epsilon$.  This
implies that $X = d(\exp_X)(1) = d(e_X)(1) \in \fp_\epsilon$.  Since $X \in u_\epsilon^-$, 
we conclude $X=0$.  In other words, (\ref{eq:comp}) is injective.
\end{proof}

\begin{cor}
\label{cor:gln} 
Assume $p>2n-2$.  Then for any  $r$-dimensional subspace  $\epsilon$ of  $\gl_n$ whose (adjoint) orbit 
$\GL_n \cdot \epsilon$ is closed in $\Grass(r,\gl_n)$, the orbit map $\GL_n \to \GL_n \cdot \epsilon$ is separable.
\end{cor}
\begin{proof}
Since the standard map $\GL_n \to \PGL_n$  is a $\bG_m$-torsor, the separability 
of the orbit map for $\GL_n$ follows immediately from Theorem~\ref{thm:sep}.
\end{proof}

The following proposition, a generalization of 
\cite[7.9]{CFP2}, enables us to identify  kernel bundles
provided we know corresponding image bundles and vice versa.

\begin{prop}
\label{dual}
Retain the notation and hypotheses of Theorem 
\ref{thm:functor}.  Then there is a natural short
exact sequence of vector bundles on 
$G/H \simeq X \subset \bE(r,\fg)$
\begin{equation}
\label{duality}
\xymatrix{
0 \ar[r]& \cKer^{j,X}(M^\#)\ar[r]& (M^\# )\otimes \cO_X 
\ar[r]&  (\cIm^{j,X}(M))^\vee \ar[r]& 0.
}
\end{equation}
\end{prop}

\begin{proof}
The proof is a repetition of that of \cite[7.9]{CFP2}.   
By \cite[2.2]{CFP2}, the sequence
\begin{equation}
\label{duality1}
\xymatrix{
0 \ar[r]& \Soc^j(\epsilon^*(M^\#)) \ar[r]& M^\# 
\ar[r]& (\Rad^j(\epsilon^*M))^\# \ar[r]& 0.
} 
\end{equation}
is an exact sequence of $H$-modules.  
Applying the functor $\cL$ to (\ref{duality1}) (which 
preserves exactness by Proposition \ref{prop:cL})
 and appealing to Theorem \ref{thm:functor},
we conclude the exactness of (\ref{duality}).
\end{proof}

In the next proposition we remind the 
reader of some standard constructions 
of bundles using the operator $\cL$ in addition to
 $\gamma$ and $\delta$ mentioned in Example \ref{ex:standard}. 

\begin{prop}
\label{tan-cot} 
Let $G$ be a reductive algebraic group and  
let $P$ be a standard parabolic subgroup. 
Set $\fg = \Lie(G)$, $\fp = \Lie(P)$, and  
let $\fu$ be the nilpotent radical of $\fp$. 
\begin{enumerate} 
\item The tangent bundle of $G/P$ is 
isomorphic to $\cL_{G/P}(\fg/\fp)$,
$$\bT_{G/P} \ \simeq \ \cL_{G/P}(\fg/\fp). $$
\item Assume that $\fg$ has a nondegenerate 
$G$-invariant symmetric bilinear form 
(such as the Killing form). 
Then the cotangent bundle of $G/P$ is 
isomorphic to $\cL_{G/P}(\fu)$:
\[
\Omega_{G/P} = \bT^\vee_{G/P} \ \simeq  \ \cL_{G/P}(\fu),
\]
where $\fu$ is viewed as $P$-module via the 
restriction of the adjoint action of $P$  on $\fp$.  
\item For $G = \SL_n$ with $p \nmid n $, 
$P = P_{r, n-r}$, and $X = G/P = \Grass(r,n)$, we have 
\[
\bT_X \simeq \gamma_r^\vee \otimes \delta_{n-r}^\vee,
\quad \Omega_X \simeq \gamma_r \otimes \delta_{n-r}.
\] 
\item For $G=\Sp_{2n}, P = P_{\alpha_n}$, 
and $Y = G/P = \LG(n, V)$, we have
\[
\bT_Y \ \simeq \ \cL_{G/P}(\fg/\fp) \ 
\simeq \ S^2(\gamma_n^\vee).
\] 
Moreover, if $p>3$, then 
\[\Omega_Y \ \simeq  \ S^2(\gamma_n).\] 
\end{enumerate} 
\end{prop}

\begin{proof}
(1) See \cite[II.6.1]{Jan}.

(2) This follows from (1) together with the 
isomorphism of $P$-modules $(\fg/\fp)^\# \simeq \fu$, 
guaranteed by the existence of a nondegenerate form.  

(3) We have $\fg = \End(V)$. Let $e_1, \ldots, e_n$ 
be a  basis of $V$, and choose a linear 
splitting of the sequence $ \xymatrix{0 \ar[r]&  W \ar[r]&  
V \ar[r]& V/W \ar[r]&  0}$ sending $V/W$ 
to the subspace generated by $e_{r+1}, \ldots, e_n$ 
(see notation introduced in Example~\ref{ex:standard}). 
We have 
$$
\End(V) = \Hom(W, V/W)\oplus \Hom(V/W, W) \oplus  
\Hom(W, W) \oplus  \Hom(V/W, V/W),
$$ 
where the sum of the last three summands is a 
$P$-stable subspace isomorphic to $\fp$. Hence, 
we have an isomorphism of $P$-modules: 
$\fg/\fp \simeq \Hom(W, V/W) \simeq W^\# \otimes V/W$.  
Therefore, 
\[
\bT_X \simeq \cL_{G/P}(\fg/\fp) \simeq \cL(W^\#) 
\otimes \cL(V/W) = \gamma_r^\vee \otimes \delta_{n-r}^\vee.
\]
Consequently, 
\[\Omega_X \simeq \gamma_r \otimes \delta_{n-r},\]
since the Killing form  is 
nondegenerate on $\fsl_n$ for $p \not | n$. 

(4). In this case, $W$ is an isotropic 
subspace of $V$, and  $W \simeq (V/W)^\#$.   
Then  $\fg/\fp \simeq \Hom_{Sym}(W, V/W) 
\simeq S^2(W^\#)$. Hence,  
\[\bT_Y \simeq \cL_{G/P}(\fg/\fp) \simeq 
S^2(\gamma_n^\vee).\]
If $p>3$, there exists a non-degenerate 
$\Sp_{2n}$-invariant form   on $\fg$ (see \cite{Sel}). Hence,  
we can  dualize to obtain the last asserted isomorphism. 
\end{proof}

%%%%%%%%%%%%%%%%%%%%%%%%%%%%%
%%%%Section on explicit example of bundles on $G$-orbits%%%%%%%%%

\section{Vector bundles on $G$-orbits of $\bE(r, \fg)$: examples}
\label{sec:examples}
We now work out some specific examples of 
vector bundles on $G$-orbits of $\bE(r, \fg)$ associated to $\fg$-modules. 
Our first example is for $\GL_n$-orbits of $\bE(m, \gl_n)$.

In the following two propositions we make an assumption that 
$\GL_n \cdot \fu_{r, n-r} \equiv \Grass(r, V)$. Since $P_{r, n-r}$ 
is the reduced stabilizer of $\fu_{r, n-r}$, this assumption is equivalent 
to the separability of the orbit map  $\GL_n \to \GL_n \cdot \fu_{r, n-r}$.  
Hence, it is satisfied for $p>2n-2$ by Theorem~\ref{thm:sep}.

\begin{prop}
\label{prop:can-gen}
Let $G = \GL(V) \ \simeq \ \GL_n$ for some $n \geq 2$, and set 
$\epsilon = \fu_{r, n-r} \in \bE({r(n-r)}, \gl_n)$,
the subalgebra of all matrices with nonzero entries
only in the top $r$ rows and right-most $n-r$
columns,  for some $r<n$.  Assume that $p > 2n-2$, so that by
Corollary \ref{cor:gln} the orbit map $\phi_\epsilon: \GL_n \to \GL_n \cdot \fu_{r, n-r} \equiv X$
is separable.  Thus,  $\phi_\epsilon$ is isomorphic to the $P$-torsor
$\GL_n \to \Grass(r,n)$, where
$P = P_{r,n-r}$ is the standard parabolic of type $(r,n-r)$.

We have the following isomorphisms of 
algebraic vector bundles on $X$:
\begin{enumerate} 
\item 
$\cIm^{1,X}(V) \ \simeq  \ \cKer^{1,X}(V) \ \simeq \ \gamma_r$, \\
$ \cIm^{j,X}(V) = 0 {\text \ for \ } j > 1.$\\[-8pt] 
\item $ \cCoker^{1,X}(V) \simeq \delta_{n-r}^\vee$,\\ 
$\cCoker^{j,X}(V) = 0 {\text \ for \ } j > 1.
$\\[-8pt]
\item $\cKer^{1,X}(\Lambda^{n-1}(V)) \simeq 
\cIm^X(\Lambda^{n-1}(V)) \simeq \delta_{n-r}^\vee,$\\ 
$\cIm^{j,X}(\Lambda^{n-1}(V)) = 0 {\text \ for \ } j > 1. $ 
\end{enumerate}
\end{prop}

\begin{proof}
Choose a basis $e_1, \ldots, e_n$ for $V$
so that both $\Rad(\epsilon^*V)$ and $\Soc(\epsilon^*M)$ 
are  the subspace $W \subset V$ 
spanned by $e_1, \ldots, e_r$. That is, 
$\Rad(\epsilon^*V) = \Soc(\epsilon^*M) = W$ 
as $P_{r,n-r}$-modules 
in the notation  of Example~\ref{ex:standard}. 
Hence, Theorem~\ref{thm:functor} implies that
\[\cIm^{1,X}(V) \simeq  \cL_{G/P}(W) = \gamma_r\quad
\text{ and } \quad 
\cKer^{1,X}(V) \simeq  \cL_{G/P}(W) = \gamma_r.\]
This proves the first part of (1).  The vanishing 
$\cIm^{j,X}(V) =0$ follows immediately from the 
fact that $\Rad^j(\epsilon^*(V))=0$ for $j \geq 2$.  

Part (2) follows from the exactness of (\ref{deltagamma}).
To prove part (3), observe that $\det: \GL_n \to \GL_1$ splits
because $p > n$, so that we may view $\bE(r,\gl_n) = \bE(r,\fsl_n)$
as the orbit of $\epsilon$ under $\SL_n$.  Since
 $\Lambda^{n-1}V \ \simeq \  V^\#$ as $\SL_n$-modules, 
 part (3) follows from  Proposition~\ref{dual}.
\end{proof}

\begin{prop} \label{prop:can-gen2}
We retain the hypotheses and notation of 
Proposition \ref{prop:can-gen}.
For any positive integer $m \leq n-r$, 
\begin{enumerate}
\item
$\cIm^{m,X}(V^{\otimes m}) \ = \ \gamma_r^{\otimes m}$,
\item
$\cIm^{m,X}(S^m(V)) \ = \ S^m(\gamma_r),$
\item 
$\cIm^{m,X}(\Lambda^m(V))\ = \ \Lambda^m(\gamma_r).$
\end{enumerate}
\end{prop}

\begin{proof}
Write $\fu(\epsilon) = k[t_{i,j}]/(t_{i,j}^p),$ 
$1 \leq i \leq r,$ $r+1 \leq j \leq n$.
The action of  $t_{i,j}$  on $V$ is given by 
the rule $t_{i,j}e_j = e_i$ and $t_{i,j} e_\ell 
= 0$ for $\ell \neq j$. Let $W = \Rad(\epsilon^*V)$ 
as in the proof of Prop.~\ref{prop:can-gen}. 
On a tensor product $M \otimes N$ of 
modules the action is given by $t_{i,j}(v \otimes w) =
t_{i,j}v \otimes w + v \otimes t_{i,j} w$; thus
$\Rad^m(\epsilon^*(V^{\otimes m}))$ is contained in the 
subspace of $V^{\otimes m}$ spanned by all elements
$e_{i_1} \otimes \dots \otimes e_{i_m}$, where 
$1 \leq i_1, \dots, i_m \leq r$, which is 
$W^{\otimes m}$. On the other hand, for 
any sequence $i_1, \dots, i_m$, with
$1 \leq i_1, \dots, i_m \leq r$, we have that 
\begin{equation}
\label{eq:tensor}
e_{i_1} \otimes \dots \otimes e_{i_m} \ = \ (t_{i_1,r+1} \dots
t_{i_m, r+m}) (e_{r+1} \otimes \cdots \otimes e_{r+m})
\end{equation}
since $r+m \leq n$.
Hence, $\Rad^m(\epsilon^*(V^{\otimes m})) = 
 W^{\otimes m}$. Therefore, the equality 
$\cIm^{m,X}(V^{\otimes m}) \ = \cL_{G/P}(W^{\otimes m}) \ 
= \ \gamma_r^{\otimes m}$
follows from Proposition \ref{prop:cL}.3,  
Theorem~\ref{thm:functor}, and Example~\ref{ex:standard}.

To show (2),  note that the action of $\fu(\epsilon)$ on   $S^m(V)$
is induced by the action on $V^{\otimes m}$ via 
the projection $V^{\otimes m} \twoheadrightarrow S^m(V)$. 
Hence, the formula \eqref{eq:tensor} is still 
valid in $S^m(V)$, and implies the 
inclusion $S^m(W) \subset \Rad^m(\epsilon^*(S^m(V))$. 
The reverse inclusion is immediate 
just as in the tensor powers case. 
Therefore, $\Rad^m(\epsilon^*(S^m(V)))= S^m(W)$,
and we conclude the equality $\cIm^{m,X}(S^m(V)) \ = 
\ S^m(\gamma_r)$ appealing to Theorem~\ref{thm:functor}.

The proof for exterior powers is completely analogous.
\end{proof}

\begin{remark}
The restriction on $m$  in Proposition 
\ref{prop:can-gen2} is not sharp.
For example, if $n = 4$ and $r = 2$, 
then it is straight forward to see that 
$\cIm^3(V^{\otimes 3}) \simeq \gamma_2^{\otimes 3}$
provided $p >2$. On the other hand, 
if $n = 3$ and $r = 2$, then 
$\cIm(V^{\otimes 2})$ is a proper 
subbundle of $\gamma_2^{\otimes 2}$,
regardless of the prime.  
\end{remark}

We use some standard Lie-theoretic  notation for the remainder of this section. 
Let $G$ be a simple algebraic 
group and let $\Delta = \{ \alpha_1, \ldots, \alpha_n\}$ 
be a fixed set of simple roots corresponding to a fixed torus $T$ 
inside a Borel subgroup $B$.  We follow the 
convention in \cite[ch.6]{Bour} in the numbering 
of simple roots. Let $\fg = \Lie(G)$, and let $\fh = \Lie T$ 
be the Cartan subalgebra.  For a simple root 
$\alpha \in \Delta$, we denote by $P_\alpha$, 
$\fp_\alpha$ the corresponding standard maximal 
parabolic subgroup and its Lie algebra.

We provide  a calculation analogous to Proposition~\ref{prop:can-gen} 
and \ref{prop:can-gen2} for the  symplectic group $\Sp_{2n}$.
The only maximal parabolic subgroup  $P_\alpha$ in standard form whose
unipotent radical is abelian corresponds to 
the longest simple root:  namely, $P =  P_{\alpha_n}$ for 
$\alpha_n$ the unique long simple root. Equivalently, 
$P_{\alpha_n}$ is the unique cominuscule parabolic
subgroup of $\Sp_{2n}$ in standard form, as 
in Definition~\ref{comin}.

If we view $G=\Sp_{2n}$ as the group of automorphisms of a 
symplectic vector space $V$ of dimension $2n$ with 
chosen symplectic basis $\{ x_1,\ldots,x_n,y_n,\ldots,y_1 \}$, 
then $P_{\alpha_n}$ is the stabilizer of the totally 
isotropic subspace spanned by $\{ x_1,\ldots,x_n \}$.

We recall from \cite[2.12]{CFP3} that $m =$ 
$n+1\choose 2$ is the dimension of each maximal
elementary subalgebra of $\fsp_{2n}$. For $p>4n-2$, we have an isomorphism  
\begin{equation}
\label{lagrange}
\bE(m, \fsp_{2n}) \ \simeq 
\ \Sp_{2n}/P_{\alpha_n} \ = \ \LG(n, V)
\end{equation}
(as follows from Theorem~\ref{thm:sep} and \cite[2.5, 2.9]{CFP3}) where $\LG(n, V)$ is the Lagrangian Grassmannian 
of maximal isotropic subspaces of the defining representation $V$.

\begin{prop}
\label{symplec}
Consider  $G = \Sp_{2n}$ and its defining 
representation $V$ (of dimension $2n$); assume $p > 4n-2$.
Let $P_{\alpha_n} \subset \Sp_{2n}$ be the maximal 
parabolic subgroup in standard form corresponding to the
longest root as described above, and let 
$\fp = \Lie(P_{\alpha_n})$.
Let $\epsilon$ be the nilpotent radical of 
$\fp$, an elementary subalgebra
of $\fsp_{2n}$ of dimension $m =$ 
$n+1\choose 2$.  As in (\ref{lagrange}), 
let 
$$Y = \bE(m, \fsp_{2n})\ \simeq \LG(n, V),$$ 
and let $\gamma_n \ \subset \ 
\cO_Y^{\oplus 2n}$ be the canonical subbundle of rank $n$.
We have the following natural identifications 
of algebraic vector bundles on $Y$:

\vspace{0.05in}
\begin{enumerate}
\item  
$\cIm^{1,Y}(V) \ \simeq \ \gamma_n, 
\quad \cIm^{j,Y}(V) = 0 {\text \ for \ } j > 1$. \\
\item 
$\cIm^{1,Y}(\Lambda^{2n-1}(V)) \simeq \gamma_n^\vee, 
\quad \cIm^{j,Y}(\Lambda^{2n-1}(V)) = 0 {\text \ for \ } j > 1. $ \\
\item 
For $m \leq n$,
\begin{enumerate}
\item
$\cIm^{m,Y}(V^{\otimes m}) = (\gamma_n)^{\otimes m}$, 
\item
$\cIm^{m,Y}(S^m(V)) = S^m(\gamma_n)$,
\item
$\cIm^{m,Y}(\Lambda^m(V))= \Lambda^m(\gamma_n).$ 
\end{enumerate}
\end{enumerate}
\end{prop}

\begin{proof}
 
We view $\Sp_{2n}$ as the stabilizer of the symplectic form
defined by the matrix 
\[
 \begin{pmatrix} 0 & I_{n}\\ -I_{n} & 0  \end{pmatrix},
 \]
so that $ \fsp_{2n}$ is the set of 
matrices of the form
\[
\begin{pmatrix} A & B \\ C &D  \end{pmatrix} 
\]
where $D = -A^{T}$ and $B$ and $C$ are $n \times n$ 
symmetric matrices.  Then $\fp \subset \fsp_{2n}$ 
is  defined by $C = 0$ (this can be easily verified   
from the explicit description of roots and roots 
spaces as in, for example, \cite[12.5]{EW06}).  
We view  $V$ as the 
space of column vectors on which these matrices
act from the left, and give $V$ the standard basis $e_1, \dots, e_{2n}$. 

The restricted enveloping algebra of $\epsilon$ 
has the form $k[t_{i,j}]/(t_{i,j}^p)$ where 
$1 \leq i \leq n$ and $n+i \leq j \leq 2n$. The generator
$t_{i,j}$ acts on $V$ by the matrix $E_{i,j}$ if 
$j = n+i$ and by $E_{i,j} + E_{j-n,i+n}$ otherwise. 
Here, $E_{i,j}$ is the matrix with 1 in the ($i,j$) position
and 0 elsewhere. Thus we have that 
\begin{equation} 
\label{eq:relations} 
t_{i,j}e_j = e_i, \quad t_{i,j} e_{i+n} = e_{j-n}, 
\quad \text{ and } \quad t_{i,j} e_\ell = 0
\end{equation}
whenever $\ell \neq j, i+n$. These 
relations immediately  imply that 
$\Rad(\epsilon^*V) = \Soc(\epsilon^*(V)) = W$ where 
$W \subset V$ is the $P$-stable subspace 
generated  by $e_1, \ldots, e_n$. 
Moreover, we also have  that $\Rad^j(\epsilon^* V) =
\Soc^j(\epsilon^* V) = 0 $ for any $j>1$. 
Applying Theorem~\ref{thm:functor}, we get
\[ \cIm^{1,Y}(V)  = \cKer^{1,Y}(V) \simeq \cL_{\Sp_{2n}/P_{\alpha_n}}(W)= 
\gamma_n, \quad \cIm^{j,Y}(V)  = \cKer^{j,Y}(V) =0 \text{ for } j>1. \]

Part (2) follows from (1) and the fact 
that $\Lambda^{2n-1}(V)$ is the dual of the
$\fg$-module $V$ (since $G$ has no non-trivial 1-dimensional
rational representation). 

We proceed to show that $\cIm^{m,Y}(V^{\otimes m}) 
\simeq (\gamma_n)^{\otimes m}$ for $m \leq n$. 
We note that as in the proof of 
Proposition \ref{prop:can-gen}, it is only necessary
to show that $(\Rad(\epsilon^*V))^{\otimes m}
\subseteq \Rad^m(\epsilon^*(V^{\otimes m}))$, since
the reverse inclusion is obvious. 

Since $\Rad^2(\epsilon^*V)=0$, the action of 
$\Rad^m(\fu(\epsilon))$  on $V^{\otimes m}$ is given by the formula
\begin{equation}
\label{formula}
(t_{i_1,n+j_1} \cdots t_{i_m, n+ j_m})
(e_{s_1} \otimes \dots \otimes e_{s_m}) =  \end{equation}
\[ 
\sum\limits_{\pi \in \Sigma_m} t_{i_{\pi(1)}, n+j_{\pi(1)}} e_{s_1} 
\otimes \dots \otimes t_{i_{\pi(1)}, n+j_{\pi(m)}} e_{s_m}.
\]
To prove the inclusion $(\Rad(\epsilon^*V))^{\otimes m}
\subseteq \Rad^m(\epsilon^*(V^{\otimes m}))$, we 
need to show that for any $m$-tuple of indices 
$(i_1, \ldots, i_m)$, $1 \leq i_j \leq n$, we have
$e_{i_1} \otimes \dots \otimes e_{i_m} \in 
\Rad^m(\epsilon^*(V^{\otimes m}))$. We first show the following

{\it \underline{Claim}.} For any simple tensor 
$e_{i_1} \otimes \dots \otimes e_{i_m}$ in 
$(\Rad(\epsilon^*V))^{\otimes m}$, there 
exists a permutation $w \in \Sigma_m$ such that 
$e_{w(i_1)} \otimes \dots \otimes e_{w(i_m)} 
\in \Rad^m(\epsilon^*(V^{\otimes m}))$.

We proceed to prove the claim. 
Let $e_{i_1} \otimes \dots \otimes e_{i_m}$ be any 
simple tensor in $(\Rad(\epsilon^*V))^{\otimes m}$. 
Applying a suitable permutation 
$\pi \in \Sigma_m$ to $(1, \ldots, m)$, we may 
assume that $(i_1, \ldots, i_m)$  has the form 
$(i_1^{a_1}, i_2^{a_2}, \ldots, i_\ell^{a_\ell})$ 
where $i_1 > i_2 > \cdots > i_\ell$ and 
$a_1+\ldots + a_\ell = m$. Applying yet another permutation,
 we may assume that the string of 
indices $(i_1, \ldots, i_m)$ has the form
\[
(i_1, i_2,\ldots, i_\ell, i_1^{a_1-1}, 
\ldots, i_\ell^{a_\ell-1}),
\]
with $i_1>i_2> \cdots >i_\ell$.  To this string  
of indices we associate the string of indices 
$j_1, \ldots, j_m$ by the following rule: 
\[j_1 = i_1, j_2 = i_2, \ldots, j_\ell = i_\ell\]
and $(j_{\ell+1}, \ldots, j_m)$ is a subset of 
$m-\ell$ distinct numbers from $\{1, \ldots, n \} 
\backslash \{i_1, i_2, \ldots, i_\ell\}$. 
We claim that 
\begin{equation}
\label{eq:perm}
(t_{i_1,n+j_1} \cdots t_{i_m, n+ j_m})
(e_{n+j_1} \otimes \dots \otimes e_{n+j_m}) = e_{i_1} 
\otimes \dots \otimes e_{i_m}. 
\end{equation}  
Indeed, relations \eqref{eq:relations} imply that 
$t_{i_1, n+j_1}e_{n+j_1} \otimes \cdots \otimes 
t_{i_m, n+j_m}e_{n+j_m} =  e_{i_1} \otimes \dots 
\otimes e_{i_m}$. We need to show that all the 
other terms in \eqref{formula} are zero.   
To have $t_{i_s, n+j_s}e_{n+j_r} \not = 0$, 
we must have either $j_s=j_r$ or $i_s=j_r$. 
By the choice of $(j_1, \ldots, j_m)$, the second 
condition $i_s = j_r$ implies that $s=r$ and, 
hence,  $j_s=j_r$. Therefore, 
$t_{i_s, n+j_s}e_{n+j_r} \not = 0$ if and only if $j_s=j_r$. 
Since by construction all 
$(j_1,\ldots, j_m)$ are distinct, we conclude that 
$t_{i_{\pi(1)}, n+j_{\pi(1)}} e_{n +j_1} \otimes 
\dots \otimes t_{i_{\pi(1)}, n+j_{\pi(m)}} 
e_{n + j_m} \not = 0$ if and only if $\pi$ is 
the identity permutation which proves \eqref{eq:perm}.  
This finishes the proof of the claim. 

Now let $e_{i_1} \otimes \dots \otimes e_{i_m}$ be 
an arbitrary tensor with $1 \leq i_j \leq n$. As we 
just proved, there exist $w \in \Sigma_m$ and  
indices $j_1, \ldots, j_m$ such that 
\begin{equation}
\label{eq:perm2}(t_{w(i_1),n+j_1} \cdots t_{w(i_m), n + j_m})
(e_{n+j_1} \otimes \dots \otimes e_{n+j_m}) = 
e_{w(i_1)} \otimes \dots \otimes e_{w(i_m)}.
\end{equation} 
The formula \eqref{formula} implies that if we apply 
$w^{-1}$ to \eqref{eq:perm2} we get the desired result, that is 
\[ 
(t_{i_1,n+w^{-1}(j_1)} \cdots t_{i_m, n+ w^{-1}(j_m)})
(e_{n+w^{-1}(j_1)} \otimes \dots \otimes e_{n+w^{-1}(j_m)}) = 
e_{i_1} \otimes \dots \otimes e_{i_m}. 
\]
Therefore, $e_{i_1} \otimes \dots \otimes e_{i_m} 
\in \Rad^m(\epsilon^*(V^{\otimes m}))$.
The statement for symmetric and exterior powers 
follows just as in Proposition~\ref{prop:can-gen2}.
\end{proof}

\begin{defn}
\label{comin}
For $\alpha$ a simple root, the (maximal) 
parabolic $P_{\alpha}$  of $G$
is called {\it cominuscule} if $\alpha$ enters with 
coefficient at most $1$ in any positive root.   
\end{defn} 
The cominuscule parabolics appear naturally in our 
study of elementary subalgebras because of the 
following equivalent description.

\begin{lemma}
\label{le:comm}{\cite[Lemma 2.2]{RRS}}
Let $G$ be a simple algebraic group and $P$ be a 
proper standard parabolic subgroup. Assume 
$p\not = 2$ whenever $\Phi(G)$ has 
two different root lengths. 
Then the nilpotent radical  of $\fp = \Lie(P)$ 
is abelian if and only if $P$ is a cominuscule parabolic.
\end{lemma}

The following is a complete list of cominuscule 
parabolics for simple groups (see, for 
example, \cite{BL} or \cite{RRS}):
\begin{enumerate}
\item Type $A_n$. $P_\alpha$ for any 
$\alpha \in \{ \alpha_1, \ldots, \alpha_n\}$. 
\item Type $B_n$. $P_{\alpha_1}$. 
\item Type $C_n$. $P_{\alpha_n}$ ($\alpha_n$ 
is the unique long simple root).
\item Type $D_n$. $P_{\alpha}$ for 
$\alpha \in \{\alpha_1, \alpha_{n-1}, \alpha_n\}$.
\item Type $E_6$. $P_{\alpha}$ for 
$\alpha \in \{\alpha_1, \alpha_6\}$.
\item Type $E_7$. $P_{\alpha_7}$. 
\end{enumerate}
For types $E_8, F_4, G_2$ there are no cominuscule parabolics.

\vspace{0.1in}

\begin{prop}
\label{prop:comin}
Let $G$ be a simple algebraic group and 
$P_\alpha$ be a maximal parabolic
subgroup of $G$.  Denote by $\fp$ the Lie algebra 
$\Lie(P_\alpha)$ and by $\fu$ the nilpotent radical
of $\fp$.
\begin{enumerate}
\item  If $G$ has type $B$ or $C$, assume 
that $p \not= 2, 3$.  Then $\fu$ 
is an elementary subalgebra if and 
only if $P_\alpha$ is cominuscule.
\item  Assume $p \not= 2$.   Then $[\fu,\fp] \ = \ \fu$.
\item  If $P_\alpha$ is cominuscule, 
then $\fp \ = \ [\fu,\fg]$.
\end{enumerate}
\end{prop}

\begin{proof}
 If $\fu$ is elementary then, in particular, it is 
abelian and, hence, $P_\alpha$ is cominuscule by 
 \cite[2.2]{RRS}.  Conversely, assume $P_\alpha$ is 
 cominuscule.   Then $\fu$ is abelian by  \cite[2.2]{RRS}. 
 Since $G$ is a simple algebraic group, 
 each root space $g_\alpha$  is one dimensional generated by 
 a root vector $x_\alpha$.  We have $x_\alpha^{[p]} =0$ since 
 $p\alpha$ is not a root. Since root vectors generate $\fu$ and 
 $\fu$ is abelian, we conclude that the $[p]$-th power operation is trivial on $\fu$. 
Hence, $\fu$ is elementary. 

To prove (2), observe that because $\fu$ is a  
Lie ideal in $\fp$, we have $[\fu, \fp]\subset \fu$. 
By the structure theory for classical Lie algebras, 
for any $\alpha \in \Phi^+$ 
there exists $h_\alpha \in \fh$ such that 
$[h_\alpha, x_\alpha] = 2x_\alpha$. Hence, 
$\fu = [\fh, \fu] \subset [\fp, \fu]$.

Finally, we proceed to prove (3).
 Let $P=P_{\alpha_i}$, let 
$I = \Delta \backslash \{\alpha_i \}$ and 
let $\Phi_I \subset \Phi$ 
 be the root system corresponding to the 
subset of simple roots $I$.  
 We have $\fg = \fp \oplus \fu^-$ where 
$\fu^- = \sum\limits_{\beta \in \Phi^+ 
\backslash \Phi^+_I} kx_{-\beta}$.  
Note that $ \Phi^+ \backslash \Phi^+_I$ 
consists of all positive roots into 
which $\alpha_i$ 
 enters with coefficient $1$. Let 
$\beta \in \Phi^+ \backslash \Phi^+_I$ and let $\gamma$ 
 be any root. If $\beta + \gamma$ is a 
root, then $\alpha_i$ enters into $\beta+\gamma$ 
 with coefficient $0$ or $1$. Therefore, 
$x_{\beta + \gamma} \not \in   \fu^-$. 
 Hence, $[x_\beta, x_\gamma] \in \fp$. 
Since $x_\beta$ for $\beta 
\in  \Phi^+ \backslash \Phi^+_I$ generate 
$\fu$, we conclude that $[\fu, \fg] \subset \fp$. 

For the opposite inclusion, we first 
show that $\fh \subset [\fu, \fg]$. 
Let $S \subset \Phi^+ \backslash \Phi^+_I$ 
be the set of all positive 
roots of the form $a_1\alpha_{1} + \ldots + a_n\alpha_n$ such that 
$a_i = 1$ and $a_j \in \{0,1\}$ for all $j \neq i$.  
For any subset $J \subset \Delta$ of simple 
roots such that the subgraph of the Dynkin 
diagram corresponding to $J$ is connected, 
we have that $\sum\limits_{\alpha_j \in J} \alpha_j$ 
is a root (\cite[VI.1.6, Cor. 3 of Prop. 19]{Bour}). 
This easily implies that for any simple root 
$\alpha_j$, $ j \not = i$, we can find $\beta_1, \beta_2 \in S$ 
such that $\beta_2 - \beta_1=\alpha_j$. Hence, 
$\{ \beta \}_{\beta \in S}$  generate the integral
root lattice $\Z\Phi$. Consider the simply laced 
case first (A, D, E). Since the  bijection 
$\alpha \to \alpha^\vee$ is linear in this 
case, we  conclude that $\{ \beta^\vee \}_{\beta \in S}$  
generate the integer coroot lattice $\Z\Phi^\vee$. 
This, in turn, implies that $\{ h_\beta \}_{\beta \in S}$ 
generate the integer form $\Lie(T_\Z)$  of the 
Lie algebra $\Lie(T) =\fh$ over $\Z$, and, therefore, 
generate $\fh = \Lie(T_\Z) \otimes_\bZ k$ 
over $k$ (see \cite[II.I.11]{Jan}).  

In the non-simply laced case (B or C), the 
relation $\beta_1 - \beta_2 = \alpha_j$ leads to 
$c_1\beta^\vee_1 - c_2\beta^\vee_2 = c_3\alpha^\vee_j$   
where $c_1, c_2, c_3 \in \{1,2\}$. 
Hence, in this case  $\{ \beta \}_{\beta \in S}$  
generate the lattice $\Z[\frac{1}{2}]\Phi^\vee$. 
Since $p\not = 2$, this still implies that 
$\{ h_\beta \}_{\beta \in S}$ generate  
$\fh = \Lie(T_\Z) \otimes_\bZ k$ over $k$.   

In either case, since 
$h_{\beta} = [x_{\beta}, x_{-\beta}] \in [\fu, \fg]$ 
for $\beta \in S$, we conclude that 
$\fh \subset [\fu, \fg]$.  

The inclusion $\fh \subset [\fu, \fg]$ 
implies $[\fp, \fh]  \subset [\fp, [\fu, \fg]]$. Hence, 
by the Jacobi identity, we have 
\[ 
[\fp, \fh]  \subset [\fp, [\fu, \fg]] = 
[[\fp, \fu], \fg]] + [\fu, [\fp, \fg]] = 
[\fu, \fg] + [\fu, \fp] \subset  [\fu, \fg].
\]
Consequently,  $\fp = [\fp, \fh] + \fh \subset [\fu, \fg]$.
\end{proof}

We next show how to realize  the tangent 
bundle of $G/P$ for a  
cominuscule parabolic $P$ of a simple 
algebraic group $G$ as a cokernel bundle.

\begin{prop}\label{prop:tan}
Let $G$ be a simple algebraic group, and let 
$P$ be a cominuscule parabolic subgroup of $G$. 
Set $\fg = \Lie(G)$, $\fp = \Lie(P)$, and let 
$\fu$ be the nilpotent radical of $\fp$. 
Assume that $G\cdot \fu  \subset \bE(\dim (\fu),\fg)$ is isomorphic to
$G/P$ (for example, assume $p$ satisfies the conditions of Theorem~\ref{thm:sep}).  
We have isomorphisms of vector bundles on $G/P$:
\[
\cIm^{1,G/P}(\fg) \ \simeq \ \cL_{G/P}(\fp)
\] 
and
\[
\cCoker^{1,G/P}(\fg) \ \simeq \ \bT_{G/P}.
\]
\end{prop}

\begin{proof} Let $X = G/P$, and let $\epsilon = \fu$. Then 
$\Rad(\epsilon^*\fg)=[u, \fg] = \fp$ 
by Prop.~\ref{prop:comin}.  
Theorem~\ref{thm:orbit2} and 
Proposition \ref{prop:cL} give an isomorphism 
$$
\cIm^{1,X}(\fg) \ \simeq \ \cL_{G/P}(\fp)
$$ 
as bundles on $X$.  
Applying Proposition \ref{prop:cL} again, 
we conclude that the short exact 
sequence of rational $P$-modules \
$0 \to  \fp \to \fg \to \fg/\fp \to 0$ \
determines a short exact sequence of bundles on $X$:
$$
0 \to \cL_{G/P}(\fp) \ \to \ \fg \otimes \cO_X \ 
\to \ \cL_{G/P}(\fg/\fp) \to 0.
$$
Applying Proposition \ref{tan-cot}, we conclude that 
$$\cCoker^{1,X}(\fg) \ \simeq \ 
( \fg \otimes \cO_X )/\cIm^{1,X}(\fg) \simeq 
( \fg \otimes \cO_X )/ \cL_{G/P}(\fp) \simeq 
\cL_{G/P}(\fg/\fp) \simeq \bT_{G/P}.$$ 
\end{proof}

We offer some other interesting bundles coming from the adjoint
representation of $\fg$.

\begin{prop} \label{prop:cominbund} Assume $p\not = 2$.
Under the assumptions of Proposition~\ref{prop:tan}, we have
\[
\cIm^{2,G/P}(\fg) \simeq \cL_{G/P}(\fu), 
\]
where $\fu$ is  viewed
as a submodule  of $\fp$ under the adjoint action of $P$. 
\end{prop}

\begin{proof}  Let $\epsilon = \fu$. By 
Proposition \ref{prop:comin}, $\Rad^2(\epsilon^* \fg) = 
[\fu, [\fu, \fg]] = \fu$. 
Hence,    $\cIm^{2,G/P}(\fg) \simeq \cL_{G/P}(\fu)$ 
by Theorem~\ref{thm:functor}.
\end{proof}

In the next three examples we specialize 
Proposition~\ref{prop:cominbund} to the 
simple groups of types $A$, $B$, and $C$.

\begin{ex}
\label{cot-sln} 
Let $G= \SL_n$, let $P = P_{r,n-r}$ be the standard 
maximal parabolic corresponding to the simple root 
$\alpha_r$, and let $X = G \cdot \fu \subset \bE(r(n-r), \fsl_n)$. 
Assume $X \simeq G/P =
\Grass(r,n)$ (e.g., $p>2n-2$).   We have an isomorphism 
of vector bundles on  $X$
\[
\cIm^{2,X}(\fg) \simeq \Omega_X 
\simeq \gamma_r \otimes \delta_{n-r}.
\] 
Indeed, this follows immediately from 
Propositions~\ref{prop:cominbund} and \ref{tan-cot}(3).
\end{ex}

\begin{ex} Let $G=\SO_{2n+1}$  be a simple 
algebraic group of type $B_n$ so that $\fg = \fso_{2n+1}$, 
and let   $P= P_{\alpha_1}$ be the standard cominuscule
parabolic subgroup of $G$.  We choose the symmetric 
form, the Cartan matrix, and the simple roots 
as in \cite[12.3]{EW06}. 
Let $\fu$ be the nilpotent radical of $\fp = \Lie(P)$,  and 
set $X = G \cdot \fu \ \subset \ \bE(2n-1,\fg)$.   Assume $p>4n-2$.  
Then $X$ is isomorphic to $G/P$  by Theorem~\ref{thm:sep}. 

We claim that
\[\cIm^{2, X}(\fg) \ = \ \cL_{G/P}(\fu) \ \simeq \ \cL_{G/P}(V_{2n-1}) \ \simeq \ \Omega_{\bP^{2n-1}}.\] 
Here, $V_{2n-1}$ is the natural module 
for the block of the Levy factor of $P$ which has type $B_{n-1}$. 
More precisely, we have $P = LU$ where 
$L$ is the Levi factor and $U$ 
is the unipotent radical.  The Levi factor 
$L$  is a block matrix group 
with blocks of size $2$ and $2n-1$.  Factoring 
out the subgroup concentrated 
in the block of size $2$, we get a simple 
algebraic group isomorphic to $\SO_{2n-1}$.   
We take $V_{2n-1}$ to be the standard 
module for this group inflated to the parabolic $P$.

To justify these claims, we note that the 
isomorphism $\cIm^{2, X}(\fg) = \cL_{G/P}(\fu)$ 
is the content of Proposition~\ref{prop:cominbund}, 
whereas  the  isomorphism 
$\cIm^{2, X}(\fg) = \cL_{G/P}(V_{2n-1})$ 
follows from an isomorphism of 
$P$-modules $\fu \simeq V_{2n-1}$ which 
can be checked by direct inspection.  The 
asserted isomorphism 
$\cIm^{2, X}(\fg) \ \simeq \ \Omega_{\bP^{2n-1}}$
follows from Proposition~\ref{tan-cot}.2, 
since the condition on $p$  guarantees 
the existence of a nondegenrate invariant form on $\fg = \Lie(G)$  
(see \cite{Sel}). 
\end{ex}

\begin{ex}
\label{cot-spn} 
Let $G = \Sp_{2n}$, $P = P_{\alpha_n}$, and 
assume that $p>4n-2$.
% does not divide $n+1$. 
We have an isomorphism of vector bundles on 
$\bE({ n+1 \choose 2}, \fg) \simeq \LG(n,V)$:
\[
\cIm^2(\fg) \simeq S^2(\gamma_n).
\]
Just as in the previous examples, this follows 
immediately from \cite[2.12]{CFP3}, \cite[2.9]{CFP3}, and Theorem~\ref{thm:sep} 
which allow us to identify
$\bE({ n+1 \choose 2}, \fg)$ with $\LG(n,V)$, 
and Propositions~\ref{prop:cominbund} and 
\ref{tan-cot}(4). Proposition~\ref{tan-cot} 
is applicable here since for $p>3$  
there exists a nondegenerate $\Sp_{2n}$-invariant 
symmetric bilinear form on $\fsp_{2n}$ 
(see \cite[p.48]{Sel}).   
\end{ex}

The following example complements 
Example~\ref{cot-sln}, evaluating 
kernel bundles rather than image bundles.

\begin{prop}
\label{prop:kernel}
Let $G=\SL_n$, and let $P=P_{r, n-r} \subset G$ 
be a cominuscule parabolic. Set $\fg = \Lie(G)$, $\fp = \Lie(P)$,  
and let $\fl$, $\fu$ be the Levi subalgebra and 
the nilpotent radical of $\fp$.   Let 
$X = G \cdot \fu \subset \bE(r,\fg)$ where $r = \dim \fu$.  
Assume $p>2n-2$ so that $X$ is
isomorphic to $G/P$. 
Then we have an isomorphism of bundles on $X \simeq G/P$:
\[
\cKer^{1,X}(\fg) \ \simeq \  \cL_{G/P}(\fu) \simeq \Omega_X. 
\] 
\end{prop} 

\begin{proof}  Let $\epsilon = \fu$ which 
is elementary by Proposition \ref{prop:comin}.
We have $\Soc(\epsilon^*(\fg)) = C_{\fg}(\fu)$, 
the centralizer of $\fu$ in $\fg$. Since $\fp$ 
is the normalizer of $\fu$, we have 
$C_{\fg}(\fu) \subset \fp$. Moreover, since $\fu \subset \fp$ 
is a  Lie ideal, so is $C_\fg(\fu)$.  
Since $\fp/\fu \simeq \fl$ is reductive, we conclude that 
$C_{\fg}(\fu)/\fu = C_{\fl}(\fu)$  belongs 
to the center of $\fl$.  We claim that this center is trivial. 

Note that in the usual matrix 
representation, $\fp$ is the set of all matrices $(a_{i,j})$
with $a_{i,j} = 0$ whenever both $i > r$ and $j \leq n-r$.
Thus $\fu = \fu_{r,n-r}$ consists of all matrices which are nonzero 
only in the upper $r$ rows and rightmost $n-r$ columns,
and $\fl$ is the algebra of all $n \times n$ matrices that
are nonzero only in the upper left $r\times r$ block and 
the lower right $(n-r) \times (n-r)$ block. So the 
center of $\fl$ consists of all matrices of the form
\[
c \ = \ \begin{pmatrix} aI_r & 0 \\ 0 & bI_{n-r} \end{pmatrix}
\]
were $a, b \in k$ have the property that $ra+(n-r)b = 0$ 
and $I_r$ is the $r \times r$ identity matrix. If 
$x = \begin{pmatrix} 0 & X \\ 0 & 0 \end{pmatrix} \in \fu$, 
where $X$ is an $r \times (n-r)$  block, then 
an easy calculation yields that $[c,x] = cx-xc = (a-b)x$.
Since $p$ does not divide $n$, we conclude that $\ell$ does  
has a trivial center, and, hence,  $\dim C_\fl(\fu)=0$. 

Since $\fp = \fl \oplus \fu$, and no elements  in $\fl$ centralize $\fu$, we conclude that 
 $C_{\fg}(\fu) \simeq u \oplus C_\fl(\fu)$.  
  Combining this  
with Theorem~\ref{thm:functor}  we get the following isomorphisms:
\begin{align*}\Ker^{1,G/P}(\fg) \simeq  
\cL_{G/P}(C_{\fg}(\fu)) \simeq &\cL_{G/P}(\fu). 
\end{align*}

Since our assumption on $p$ implies that the Trace form on $\fsl_n$ 
is non-degenerate, we conclude that $\cL_{G/P}(\fu)$ is isomorphic to the cotangent 
bundle $\Omega_{G/P}$ by Proposition~\ref{tan-cot}.
\end{proof}

%%%%%%%%%%%%%%%%%%%%%%%%%%%%%%%%
%%%%%%%Section on semi-direct products%%%%%%%%%%%%%
%%%%%%%%%%%%%%%%%%%

\section{Vector bundles associated to semi-direct products}
\label{sec:Gsemi}

In this section, we provide a reinterpretation of 
``$\GL$-equivariant $kE$-modules" considered in 
\cite{CFP2} as modules for the subgroup scheme 
$G_{(1),n} = \bV_{(1)} \rtimes\GL_n$ of the algebraic group 
$\bV\rtimes \GL_n$ of Example 1.10 of \cite{CFP3}.   This leads 
to consideration of rational representations for 
semi-direct product subgroup schemes 
$\bW_{(1)} \rtimes H$ of the affine algebraic group 
$\bW \rtimes H$, where $H$ is any affine algebraic group 
and $W$ is any faithful rational $H$-representation.

The representations of $G_{(1),n}$ and 
$\bW_{(1)} \rtimes H$ we consider do not typically 
extend to the algebraic 
groups $\bV\rtimes \GL_n$ and $\bW \rtimes H$.

\begin{notation}
\label{not:V}
Throughout this section, $V$ is an $n$-dimensional 
vector space with chosen basis, so that
we may identify $\GL(V)$ with $\GL_n$ and $V$ 
with the defining representation of $\GL_n$.  Let 
$\bV = \Spec(S^*(V^\#)) \simeq \bG_a^{\oplus n}$ 
be the vector group associated to $V$, and let 
$\bV_{(1)} \simeq (\bG_{a(1)})^{\oplus n}$ be the 
first Frobenius  kernel of $\bV$. The standard 
action of $\GL_n$ on $V$ induces an action on the 
vector group $\bV$. Moreover, it is straightforward 
that this action stabilizes the subgroup scheme 
$\bV_{(1)} \subset \bV$. Hence, we can form the following 
semi-direct products:  
\begin{equation}
\label{eq:V}
\xymatrix@-0.9pc{G_{1,n} \ar@{=}[r]^-{\rm def}& 
\bV \rtimes \GL_n & \quad G_{(1),n}  
\ar@{=}[r]^-{\rm def}& \bV_{(1)} \rtimes\GL_n.}
\end{equation}
Let 
\begin{equation}
\label{eq:g}
\xymatrix@-0.9pc{\fg_{1,n}\ar@{=}[r]^-{\rm def}&  
\Lie(G_{(1),n}) = \Lie(G_{1,n})}.
\end{equation}
We view $V \simeq \Lie(\bV_{(1)}) \subset \fg_{1,n}$  
as an elementary subalgebra of $\fg_{1,n}$ which is 
also a Lie ideal stable under the adjoint action of $G_{1,n}$.  
\end{notation}

For any $r$-dimensional subspace $\epsilon  \subset V 
\subset \fg_{1,n}$, we consider the adjoint 
action of $G_{1,n}$ on $\epsilon$.   Here, $V$ is stable under 
the adjoint action, and the action of $\bV$ on $V$ is trivial.
Moreover, the restriction of this adjoint action on $\epsilon$
to $\GL_n \subset G_{1,n}$ can be identified with the 
action of $\GL_n$ on $\epsilon \in \Grass(r,n)$ determined by
left multiplication by $n\times n$ matrices on a column vector.
This left multiplication map $\GL_n \to \Grass(r,n)$ is locally a 
product projection and thus separable.  Thus, the orbit map
$\phi_{\epsilon}: G_{1,n} \to \bE(r,\fg_{1,n})$ can be identified 
with the composition 
$$G_{1,n}\  \ \to \GL_n \ \to \ \Grass(r,n) \subset \ \bE(r,\fg_{1,n})$$
and thus induces an isomorphism 
$$G_{1,n} \cdot \epsilon \simeq \Grass(r,n).$$
In particular, the orbit map restricted to $\GL_n$, 
$\phi_\epsilon: \GL_n \ \to \ G_{1,n}\cdot \epsilon$, is
separable.

\vskip .2in

We recall the notion of a $\GL$-equivariant 
$kE$-module considered in \cite{CFP2}.

\begin{defn}
Let $E$ be an elementary abelian $p$-group 
of rank $n$ and choose some linear 
map  $V \to \Rad(kE)$  such that 
the composition $V \to \Rad(kE) \to \Rad(kE)/\Rad^2(kE)$  is an isomorphism.  
This determines an identification 
$kE \simeq S^*(V)/\langle v^p,v\in V \rangle$.  
Then $M$ is said to be a {\it $\GL$-equivariant
$kE$-module} (in the terminology of  \cite[3.5]{CFP2}) 
if $M$ is provided with two pairings
\begin{equation}
\label{pair} 
S^*(V)/\langle v^p,v\in 
V \rangle \otimes M \ \to \ M, 
\quad \GL(V) \times M \to M
\end{equation}
such that the second pairing makes $M$ into a rational $\GL(V)$-module 
and the first pairing is $\GL(V)$-equivariant 
with respect to the diagonal
action of $\GL(V)$ on  $S^*(V)/\langle v^p,
v\in V \rangle \otimes M$.
\end{defn}

As the next proposition explains, the consideration 
of $\GL$-equivariant $kE$-modules
has a natural interpretation as $G_{(1),n}$-representations 
for $G_{(1),n} = \bV_{(1)}\rtimes \GL_n$.

\begin{prop}
\label{prop:semi}
There is a natural equivalence of categories between 
the category of rational modules for the
group scheme $G_{(1),n}$ and the category of 
``$\GL$-equivariant $kE$-modules".
\end{prop}

\begin{proof} 
Assume that we are given a functorial action 
of the semi-direct product 
$$(G_{(1),n})(A) = \bV_{(1)}(A) \rtimes \GL_n(A)\quad 
\text{on~}  M\otimes A$$
as $A$ runs over commutative $k$-algebras.
We view this as a group action of pairs 
$(v,g) = (v,1)\cdot (0,g)$ on $M$.  Since
$(0,g) \cdot (v,1) = ({}^gv,g) = ({}^gv,1)\cdot (0,g)$ 
in the semi-direct product, we conclude 
for any $m \in M$ that the action of $(0,g)$ on  
$(v,1)\circ m$ equals the action of $({}^gv,1)$ on $(0,g)\circ m$.
This is precisely the condition that the action 
of $\bV_{(1)} \times M \to M$ is $\GL_n$-equivariant for
the diagonal action of $\GL_n$ on $\bV_{(1)} \times M$.
Consequently, once the identification 
$kE \simeq k\bV_{(1)} = \fu(\Lie(\bV))$ is chosen, 
giving a $\GL_n$-equivariant action 
$kE\times M \to M$ is the same as giving actions of 
$\bV_{(1)}$ and $\GL_n$ on $M$ which
satisfy the condition that this pair of actions 
determines an action of the semi-direct product.

Conversely, given a $\GL_n$-equivariant $kE$-module $N$, 
it is straightforward to check
that the actions of $\GL_n$ and $kE \simeq k\bV_{(1)}$ 
determine an action of $G _{(1),n}$
on the underlying vector space of $N$.
\end{proof}

Observe that we have $\GL_n$ acting on $\fg_{1,n}$ by  
restricting the adjoint action of $G_{1,n}$ on its Lie algebra
to $\GL_n \ \subset \ G_{1,n}$. 
This, in turn, makes $\bE(r,\fg_{1,n})$ into a $\GL_n$-variety. 
We next observe that rational $G_{(1),n}$-representations 
(even those which are not
restrictions of $G_{1,n}$-representations) 
lead to $\GL_n$-equivariant sheaves on $\Grass(r,V)$.

For the rest of this section, we will require slight 
generalizations of several statements occurring earlier in the paper. 
the proofs of these generalizations are identical to the ones used 
to show the original conclusions.

\begin{remark}
\label{rem:gen} 
Let $\wt G$ be an affine group scheme, 
let $\fg = \Lie(\wt G)$,  and let $G \hookrightarrow \wt G$ 
be a closed, reduced algebraic subgroup of $\wt G$. 
Let $\epsilon \in \bE(r, \fg)$, and let 
$X = G \cdot \epsilon \subset \bE(r, \fg)$ 
be the orbit of  $\epsilon$ under the action of 
$G$ on $\bE(r, \fg)$ induced by the adjoint 
action of $\wt G$ on $\fg$.  Let $M$ a  rational 
$\wt G$-module.   Then proofs of 
Corollary~\ref{cor:equiv}, Theorem~\ref{thm:orbit2}, 
and Theorem~\ref{thm:functor} apply to prove the corresponding
statements for 
$\cIm^{j, X}(M)$, $\cKer^{j, X}(M)$, and $\cCoker^{j, X}(M)$
in this slightly modified context.

In particular, the aforementioned results are applicable to the 
situation $\wt G = G_{1,n}$ and $G = \GL_n$.
\end{remark}

Using the $\GL_n$ equivariance of image and kernel 
sheaves, we obtain the following comparison. 

\begin{thm}
\label{thm:identify1}
Let $M_{|E}$ denote a $kE$-module associated 
to a rational $G_{(1),n}$-module $M$.
Choose some $r, 1 \leq r < n$, and some $j$ 
with $1 \leq j \leq (p-1)r$.  
Let $\epsilon \subset V \subset \fg_{1,n}$ be 
an $r$-dimensional subspace.  Then there are natural
identifications of $\GL_n$-equivariant vector bundles 
on $X = \GL_n \cdot \epsilon \simeq \Grass(r,V)
\ \subset \bE(r,\fg_{1,n})$,
$$
\cIm^{j,X}(M) \simeq  \cIm^j(M_{|E}), \quad 
\cKer^{j,X}(M) \simeq \cKer^j(M_{|E}),
$$
where the vector bundles 
$\cIm^j(M_{|E}), \ \cKer^j(M_{|E})$ on $\Grass(r,V)$ 
are those constructed in \cite{CFP2}.
\end{thm}

\begin{proof}
The vector bundle $\cIm^{j,X}(M)$ on 
$X \subset \bE(r,\fg_{1,n})$ is $\GL_n$-equivariant by 
Corollary~\ref{cor:equiv}  
with the fiber at the point $\epsilon \in X$ isomorphic to 
$\Rad^j(\epsilon^*M)$ by Theorem~\ref{thm:orbit2} (see also Remark~\ref{rem:gen}).    
As proved in \cite[7.5]{CFP2}, the vector bundle  
$\cIm^j(M_{|E})$ on $\Grass(r,V)$ is also $\GL_n$-equivariant
with fiber over $\epsilon \in \Grass(r,V)$ 
also isomorphic to $\Rad^j(\epsilon^*M)$. 
Hence, $\cIm^{j,X}(M) \simeq \cIm^j(M_{|E})$ 
by Proposition~\ref{prop:cL2}. 

The argument for the kernels is strictly analogous.  
\end{proof}

As an immediate corollary of Theorem \ref{thm:identify1}, 
we conclude the following interpretation
of the computations of \cite{CFP2}.   
The modules $N, M, R$ of the following proposition
are rational $G_{(1),n}$-modules which do not extend 
to rational $G_{1,n}$-modules.  For example, the 
$G_{(1),n}$ action on $N = S^*(V)/S^{*\geq j+1}(V)$ 
when restricted to $\bV_{(1)}$  increases 
the degree of elements of $N$, whereas the $\GL_n$ structure
is a direct sum of actions on each symmetric power 
$S^i(V)$.  See \cite[3.6]{CFP2} for details of
the $G_{(1),n}$ -- structures on $N, M, R$.

\begin{prop} \cite[7.12,7.11,7.14]{CFP2}
\label{perspect} Let $\epsilon \subset V$ be an 
$r$-plane for some integer $r$, $1\leq r\leq n$, 
and let $X = \GL_n \cdot \epsilon \simeq \Grass(r,n)$ 
be the orbit of  $\epsilon \in \bE(r,\fg_{1,n})$.
We have the following isomorphisms of $\GL_n$-equivariant 
vector bundles on $\Grass(r,n)$:

(1) For the rational $G_{(1),n}$-module 
$N = S^*(V)/S^{*\geq j+1}(V)$ and for any $j$ with $1 \leq j \leq p-1$,
\[\cIm^{j,X}(N) \ \simeq \ S^j(\gamma_r),\] 
where $\gamma_r$ is the canonical rank $r$ 
subbundle of the trivial rank $n$ bundle on $\Grass(r,n)$.

(2)  For the rational $G_{(1),n}$-module 
$M = \Rad^r(\Lambda^*(V))/\Rad^{r+2}(\Lambda^*(V))$,
\[
\cKer^{1,X}(M) \ \simeq \ \cO_X(-1) \oplus \cO_X^{{n\choose r+1}}.
\] 

(3) For the rational $G_{(1),n}$-module 
$R = S^{r(p-1)}(V)/\langle S^{r(p-1)+2}(V); 
v^p, v \in \bV \rangle$,
\[\cKer^{1,X}(R) \ \simeq \ \cO_X(1-p) \oplus 
(\Rad(R) \otimes \cO_X).\]
\end{prop}

We point out that specializing Proposition~\ref{perspect}.1 
to the case $j=1$ gives a realization of the
canonical subbundle $\gamma_r$ on 
the Grassmannian as an image bundle  of the 
$G_{(1),n}$-module $S^*(V)/S^{\geq 2}(V)$ different 
from the realization of $\gamma_r$ given in  
Proposition~\ref{prop:can-gen}.1.

Our new examples of vector bundles arise by 
considering subgroup schemes of $G_{(1),n}$
which we now introduce.

\begin{defn}
\label{defn:W}
Let $H$ be an algebraic group and $W$ a faithful, 
finite dimensional rational $H$-module of dimension $n$; 
let $\bW$ be the associated vector group ($\simeq \bG_a^{\oplus n}$) 
equipped with the action of $H$.  Let 
$$ 
G_{W,H} \ \equiv \ \bW_{(1)}\rtimes H \subset \bW\rtimes H,
$$
and let
\[
\fg_{W,H} = \Lie( G_{W,H}).
\]
For any subspace $\epsilon \subset W$ of dimension $r$, 
we identify the $\bW\rtimes H$-orbit (i.e., adjoint orbit) of $\epsilon \in
\bE(r,\fg_{W,H})$ with $Y = H\cdot \epsilon 
\subset \ \Grass(r,W) \subset \bE(r,\fg_{W,H})$,
where $H$ acts on $\Grass(r,W)$ as the
restriction of the standard quotient $\GL(W) \to \Grass(r,W)$.

If $\rho: H \to \GL_n$ defines the representation of $H$ on $W$, 
then $\rho$ induces closed embeddings
$$\\ \bW\rtimes H \ \subset \ G_{1,n}, 
\quad G_{W,H} \ \subset \ G_{(1),n}.$$
\end{defn}

We next apply Corollary~\ref{cor:equiv} together with Remark~\ref{rem:gen} 
to $\wt G = G_{W,H}$ and $G = H$ to obtain the following equivariance statement. 

\begin{cor}
\label{prop:W}
Using the notation and terminology of Definition \ref{defn:W}, let
$Y = H\cdot \epsilon \subset \ \Grass(r,W) \subset \bE(r,g_{W,H})$ 
be the $W \rtimes H$-orbit 
of some $\epsilon \subset W$, a dimension $r$ subspace 
of the rational $G_{W,H}$-module $W$.
Let $M$ be a finite dimensional rational $G_{W,H}$-module.
For any $j, \ 1 \leq j \leq (p-1)r$, 
the image and kernel sheaves
$$\cIm^{j,Y}(M), \ \cKer^{j,Y}(M)$$
on $Y$ are $H$-equivariant  algebraic vector bundles.
\end{cor}

As in Definition 5.7, we consider the adjoint action of $H$ on $\epsilon \in
\Grass(r,W)$.  Let $H_\epsilon \subset H$ denote the (reduced) stabilizer of $\epsilon$
and let 
\begin{equation}
\label{eq:phi}
\xymatrix@=14pt{\phi: H/H_\epsilon \ar[r]& Y=H \cdot \epsilon}  
\end{equation}
denote the morphism of varieties induced by the orbit map $H \to H\cdot 
\epsilon \subset \Grass(r, W)$.  We recall that $\phi$ is always a homeomorphism, and it is   
an isomorphism of varieties if the orbit map is separable.

We easily extend the computations of Proposition 
\ref{perspect} by considering the rational 
$G_{(1),n}$-modules $N$, $M$, $R$ upon restriction  
to $G_{W,H} \subset G_{(1),n}.$
If $i: Y \subset X$ is an embedding of a locally closed subvariety $Y$ in
a quasi-projective variety $X$ and if $\cE$ is
an algebraic vector bundle on $X$, then we denote 
by $\cE_{|Y}$ the restriction $i^*\cE$ of $\cE$
to $X$.  Simillarly, if $i:H \to G$ is a closed 
embedding of affine group schemes and $M$ is a 
rational $G$-module, then we denote by $M_{|H}$ 
the restriction of $M$ to $H$.

\begin{thm}
\label{thm:H}
Retain the context and notation of Definition~\ref{defn:W}, 
and assume that  the map $\phi: H/H_\epsilon \to Y=H\cdot \epsilon$ 
of \eqref{eq:phi}  is an isomorphism.
We have the following isomorphisms of 
$H$-equivariant vector bundles on $Y \ \subset \ \Grass(r,W)$:

(1) For  the rational $G_{(1),n}$-modules 
$N = S^*(W)/S^{*\geq j+1}(W)$ and any $j$ such that $1 \leq j \leq p-1$,
\[
\cIm^{j,Y}(N_{|G_{W,H}}) \ \simeq \ S^j(\gamma_r)_{|Y},\] 
where $\gamma_r$ denotes  the canonical rank $r$ subbundle on $\Grass(r,W)$.

(2)  For the rational $G_{(1),n}$-modules 
$M = \Rad^r(\Lambda^*(W)/\Rad^{r+2}(\Lambda^*(W))$,
\[
\cKer^{1,Y}(M_{|G_{W,H}}) \ \simeq \ \cO_Y(-1) \oplus \cO_Y^{{n\choose r+1}}.
\]

(3) For the rational $G_{(1),n}$-modules $R = S^{r(p-1)}(W)/\langle 
S^{r(p-1)+2}(W); v^p, v \in V \rangle$,
\[
\cKer^{1,Y}(R_{|G_{W,H}}) \ \simeq \ \cO_Y(1-p) \oplus 
(\Rad(R) \otimes \cO_Y).\]
\end{thm}

\begin{proof}  
Let $L$ be a rational representation of $G_{(1),n}$.
Theorem~\ref{thm:functor} implies that the fibers 
above $\epsilon \in Y$ of $\cIm^{j,Y}(L_{|G_{W,H}})$, 
$\cIm^{j, X}(L)_{|Y}$ are both isomorphic to  $\Rad^j(\epsilon^*L)$
as modules for $H_\epsilon \subset H$.  
Since both $\cIm^{j,Y}(L_{|G_{W,H}})$, $\cIm^{j, X}(L)_{|Y}$ are 
$H$-equivariant coherent sheaves on 
$Y \simeq H/H_\epsilon$, we conclude that they are isomorphic by 
Theorem~\ref{thm:functor}.

The first statement now follows immediately from 
Proposition~\ref{perspect} and the above observation 
applied to $N$. The proofs for (2) and (3) are completely analogous.
\end{proof}

We restate as a corollary the following special 
case of Theorem \ref{thm:H}(1)  We can
interpret this corollary as saying for any affine 
algebraic group $H$ and any subgroup
$S \subset H$ which is the stabilizer of some 
$r$-dimensional subspace $\epsilon $ of an $H$-module
$W$ that the $H$-equivariant vector 
bundle $\cL_{H/S}(\epsilon)$ on $H/S$
can be realized as $\cIm^{1,Y}(M)$ for some 
$G_{W,H}$-representation $M$.

\begin{cor}
\label{cor:im1}
Let $H$ be an affine algebraic group, and let $W$ 
be a finite dimensional rational $H$-module. 
Choose an $r$-dimensional subspace $\epsilon \subset W$, 
let $ S\subset H$ be the (reduced) stabilizer of $\epsilon$, 
and assume that $\phi: H/S \to Y$ induced by the orbit map is an isomorphism.
Then there exists a rational $G_{W,H}$-module $M$ such that 
$$
\cIm^{1,Y}(M) \ \simeq \  (\gamma_r)_{|Y} \ \simeq \ \cL_{H/S}(\epsilon)
$$ 
as $H$-equivariant algebraic vector bundles on 
$Y \subset \Grass(r,W) \subset \bE(r,g_{W,H})$.  
Here, $\gamma_r$ is the canonical rank $r$ subbundle
of $W \otimes \cO_{\Grass(r,W)}$.

\end{cor}

\begin{proof}  The isomorphism $\cIm^{1,Y}(M)  \ \simeq \ (\gamma_r)_{|Y}$
 is a special case of Theorem~\ref{thm:H}(1) for $j=1$. 

Note that the given action of $H$ on $W$ 
induces an action on $\Grass(r, W)$ and also makes 
the canonical subbundle $\gamma_r$ on $\Grass(r, W)$ $H$-equivariant. 
The action of $S$ on the fiber  of 
$\gamma_r$ (an $r$-dimensional subspace of $W$) above the point $\epsilon \in \Grass(r,W)$ is the 
restriction of the action of $H$ on $W$. 
Similarly, the action of $S$ on the fiber
of $\cL_{H/S}(\epsilon)$ above the point $eH \in H/S$ is the restriction 
 to $S$ acting on this fiber
of the action of $H$ on $W$.  Hence, 
$\cL_{H/S}(\epsilon) \simeq (\gamma_r)_{|Y}$   
by Prop.~\ref{prop:cL2}.  
\end{proof}

In the following proposition, we consider the 
evident semi-direct product $G_{\fh,H} \equiv \fh \rtimes H$ 
determined by the adjoint action of $H$ on $\Lie(H) = \fh$.

\begin{prop}
\label{prop:cot} 
Let $H$ be a simple algebraic group, and 
assume that $p>2h-2$ where $h$ is the Coxeter number for $H$. 
Let $P$ be a standard cominuscule parabolic  of $H$, 
let $\fu$ be the nilradical of $\Lie(P)$, and let $Y = H \cdot \fu \subset
\Grass(\dim(\fu),\fh)$ be the orbit of $\fu$ under the adjoint action of $H$.  

Let $\Omega_{H/P}$ be the cotangent bundle on $H/P \simeq Y$.
Then for any $j$, $1 \leq j \leq p-1$, there exists a 
rational $G_{\fh,H}$-module $N$ such that 
\[
\cIm^{j,Y}(N) \  \simeq \ S^j(\Omega_{H/P}).\]
\end{prop}

\begin{proof} Recall that $P$ is the (reduced) stabilizer of 
$\fu \subset \fg$ for the adjoint action of $G$ on $\fg$.  
Theorem~\ref{thm:sep} implies that the orbit map induces 
an isomorphism $H/P  \simeq Y$.
Let $r=\dim \fu$. By Theorem \ref{thm:H}(1), we can 
find a rational $G_{\fh, H}$-module $N$ 
such that  $\cIm^{j,Y}(N)  \simeq S^j(\gamma_r)_{|Y} = 
S^j((\gamma_r)_{|Y})$ where $\gamma_r$ 
is the canonical rank $r$  subbundle on $\Grass(r, \fh)$.  
As shown in Corollary~\ref{cor:im1}, 
$(\gamma_r)_Y \ \simeq \ \cL_{H/S}(\fu)$ which is
isomorphic to $\Omega_{Y}$ by Proposition~\ref{tan-cot}(2).
The latter is applicable since the assumption on $p$ implies that $\fh$ 
admits a nondegenerate, $H$-invariant,  symmetric bilinear form (see \cite[p.48]{Sel}). 

\end{proof} 

We apply Proposition~\ref{prop:cot} to simple groups of type $C_n$ to obtain 
 the following realization results for bundles on the 
Lagrangian Grassmannian (cf. Propositions \ref{symplec} and \ref{tan-cot}).

\begin{ex}
\label{Sp2m}
Assume $p>4n-2$. Take $H = \Sp_{2n}$, and let $\epsilon \subset \fh=\fsp_{2n}$ 
be the nilpotent radical of the Lie algebra
of the standard cominuscule parabolic subgroup 
$P_{\alpha_n} \subset H$; let $r = \dim \epsilon = {n+1\choose2}$.   
Consider $Y = H\cdot \epsilon  \simeq \LG(n,n) \simeq \bE(r,\fh)$ 
as in (\ref{lagrange}). 
Then there is a rational 
$G_{\fh, H}$-module $N$ such that  
\[
\cIm^{j,Y}(N_{|G_{\fh,H}}) \ \simeq  \ S^j(\Omega_Y)  \ \simeq  \ S^{2j}(\gamma_n)
\]
for any $j$, $ 1 \leq j\leq p-1$. 
Here, $\gamma_n$ is the canonical rank $n$ 
subbundle on $\LG(n,n)$.
\end{ex}


\begin{thebibliography}{20}

\bibitem{Bo} A. Borel, {Linear Algebraic Groups}.  Graduate Texts in Mathematics, 
126. Springer-Verlag, New York, 1991.

\bibitem{Bour} N. Bourbaki, { Groupes et algebres de Lie. Chaps. 4, 5 et 6.} Masson, Paris, 1981.

\bibitem{BL} S. Billey and V. Lakshmibai, { Singular  Loci of Schubert Varieties}, 
Progress in Mathematics {\bf 182} (2000) Birkh\"auser, Boston, 2000.

\bibitem{CFP2}  J. F. Carlson, E. M. Friedlander, and J. Pevtsova,  
Representations of elementary abelian $p$-groups and bundles on Grassmannians, 
{\em Advances in Math.} {\bf 229} (2012) 2985\--3051.

\bibitem{CFP3}  ---------, {Elementary subalgebras of Lie algebras}.  To appear.

\bibitem{DG} M. Demazure and P. Gabriel, {Groupes alg\'ebriques}. Tome I. North Holland, 1970.

\bibitem{EW06} K. Erdmann and M. Wildon, {Introduction to Lie algebras}, {\em 
Springer Undergraduate Mathematics Series}, Springer, 2006.

\bibitem{FPar1} E. M. Friedlander and B. Parshall, {Cohomology of algebraic and 
related finite groups}, {\em Invent. Math.} {\bf 74} (1983), 85\--117.

\bibitem{FP3} E. M. Friedlander and J. Pevtsova, {Constructions for infinitesimal group schemes}, 
{\em Trans. of the  AMS}, {\bf 363} (2011), no. 11, 6007\--6061.

\bibitem{FP4} ---------, {Generalized support varieties for finite group schemes},
{\em Documenta Mathematica}, Extra volume Suslin (2010), 197\--222.

%\bibitem{Fu}  W. Fulton, Young tableaux, with applications to representation theory and geometry, 
%{\em Lond. Math. Soc. Stud. Texts} {\bf 35} Cambridge University Press, Cambridge, 1997.

\bibitem{Har10} J. Harris, {Algebraic geometry. a first course}, {\em Graduate Texts in Mathematics}, 
{\bf 133} Springer, New York, 2010.

\bibitem{Har} R. Hartshorne, {Algebraic Geometry}.  {\em Graduate Texts in Mathematics} {\bf 52}   Springer, 1977.

\bibitem{Hum1}  J. Humphreys, {Introduction to Lie algebras and Representation theory}, 
{\em Graduate Texts in Mathematics} {\bf 9},  Springer, Fifth edition, 1987. 

\bibitem{Jan}  J. C. Jantzen, {Representations of algebraic groups, $2^{nd}$ edition.}
{\em Math Surveys and Monographs} {\bf 107}, American Mathematical Society, 2003.

\bibitem{Jan04} ---------, {Nilpotent orbits in Representaion Theory}, 
in {\em Lie Theory: Lie algebras and Representations}, ed. J.-P. Anker, 
B. Orsted, Birkh\"auser, Boston, (2004).  

%\bibitem{Merk} A. Merkurjev, {Equivariant K-theory}.  
%{\em Handbook of K-Theory, vol 1,2}, Springer, 2005.

\bibitem{McN} G. McNinch, {Abelian reductive subgroups of reductive groups}, J. Pure Applied 
Algebra {\bf 167} (2002), 269-300.

\bibitem{Panin} I. Panin, {On the Algebraic $K$-Theory of Twisted Flag Varieties}, 
{\em $K$-theory}, {\bf 8}, (1994) 541\--585.


\bibitem{Quillen} D. Quillen, {Higher Algebraic K-theory, I}, 
{\em Algebraic K-theory, I: Higher K-theories}, 
 Lectures notes in Math {\bf 341} (1973), 85\--147.

\bibitem{RRS} R. Richardson, G. R\"ohrle, and R. Steinberg, 
{Parabolic subgroups with Abelian unipotent radical}, 
{\em Invent. Math.} {\bf 110}, (1992), 649\--671. 

\bibitem{Sel} G. B. Seligman, { On Lie algebras of prime characteristic}, 
{\em Mem. Amer. Math. Soc.} 
{\bf 19}(1965), 1 \-- 83. 


\bibitem{Sp} T. Springer,  Linear algebraic groups, Second edition. 
{\em Progress in Mathematics,} 9. Birkh\"auser Boston, Inc., Boston, MA, (1998). 

\bibitem{SFB1} A. Suslin, E. Friedlander, and C. Bendel, 
Infinitesimal 1-parameter subgroups and cohomology, 
{\em J. Amer. Math. Soc.} {\bf 10} (1997), 693\--728.


\bibitem{Th} R. Thomason,  {Algebraic K-theory of group scheme actions}, 
{\em Algebraic topology and algebraic K-theory}, 
{\em Annals of Math Studies} {\bf 113} (1987), 539\--563.


\end{thebibliography}
\end{document}